\documentclass[11pt]{amsart} 
\usepackage{amssymb}
\usepackage{graphicx,overpic}
\usepackage{bm}
\usepackage[utf8]{inputenc}
\usepackage[T1]{fontenc}
\usepackage{mathptmx} 
\usepackage{amsthm}
\usepackage{hyperref}
\usepackage[dvipsnames]{xcolor}
\usepackage{soul}
\setstcolor{red}

\newtheorem{theorem}{Theorem}[section]
\newtheorem{corollary}[theorem]{Corollary}
\newtheorem{lemma}[theorem]{Lemma}
\newtheorem{proposition}[theorem]{Proposition}
\newtheorem{remark}[theorem]{Remark}

\renewcommand{\Re}{\operatorname{Re}}
\renewcommand{\Im}{\operatorname{Im}}

\title[Universality for counting statistics of random normal matrices]{Universality for fluctuations of counting statistics of random normal matrices}
\author[J. Marzo]{Jordi Marzo}
\address{
J. Marzo: Departament de Matem\`atiques i Inform\`atica \\
Universitat de  Barcelona \& Centre de Recerca Matemàtica \\
Gran Via 585, 08007, Barcelona \\
 Spain}
\email{jmarzo@ub.edu}

\author[L.D. Molag]{Leslie Molag}
\address{L.D. Molag: Mathematics Department \\ Carlos III University of Madrid\\ Avda. de la Universidad, 30. 28911 Leganés\\ Spain}
\email{lmolag@math.uc3m.es }

\author[J. Ortega-Cerdà]{Joaquim Ortega-Cerd\`a}
\address{
J. Ortega-Cerd\`a: Departament de Matem\`atiques i Inform\`atica \\
Universitat de  Barcelona \& Centre de Recerca Matemàtica \\
Gran Via 585, 08007, Barcelona \\
 Spain}
\email{jortega@ub.edu}

\thanks{JM and JOC are supported by grants PID2021-123405NB-I00 and CEX2020-001084-M by the 
Agencia Estatal de Investigación. \\
LM is supported by the UC3M grant 2024/00002/007/001/023 ``Local and global limits of complex-dimensional DPPs" and the grant ID2024-155133NB-I00,
``Orthogonality, Approximation, and Integrability:
Applications in Classical and Quantum Stochastic Processes (ORTH-CQ)" by the Agencia Estatal de Investigación.\\
JOC is also supported by ICREA Academia of the Departament de Recerca i Universitats.}

\begin{document}

\begin{abstract}
    We consider the fluctuations of the number of eigenvalues of $n\times n$ random normal matrices depending on a potential $Q$ in a given set $A$. The eigenvalues of random normal matrices are known to form a determinantal point process, and are known to accumulate on a compact set called the droplet under mild conditions on $Q$. When $A$ is a Borel set strictly inside the droplet, we show that the variance of the number of eigenvalues $N_A^{(n)}$ in $A$ has a limiting behavior given by
    \begin{align*} 
\lim_{n\to\infty} \frac1{\sqrt n}\operatorname{Var } N_A^{(n)} = \frac{1}{2\pi\sqrt\pi}\int_{\partial_* A} \sqrt{\Delta Q(z)} \, d\mathcal H^1(z),
\end{align*}
where $\partial_* A$ is the measure theoretic boundary of $A$, $d\mathcal H^1(z)$ denotes the one-dimensional Hausdorff measure, and $\Delta = \partial_z \overline{\partial_z}$. We also consider the case where $A$ is a microscopic dilation of the droplet and fully generalize a result by Akemann, Byun and Ebke for arbitrary potentials. In this result $d\mathcal H^1(z)$ is replaced by the harmonic measure at $\infty$ associated with the exterior of the droplet. This second result is proved by strengthening results due to Hedenmalm-Wennman and Ameur-Cronvall on the asymptotic behavior of the associated correlation kernel near the droplet boundary.
\end{abstract}

\maketitle

\section{Introduction}

Given $Q:\mathbb C\to \mathbb R$ the associated \emph{random normal matrix} model consists of all complex $n\times n$ normal matrices $M$, equipped with the probability measure
\begin{align} \label{eq:defJPDFM}
    d\mathcal P_n(M) = \frac{1}{\mathcal Z_n} e^{-n\operatorname{Tr} Q(M)} dM,
\end{align}
where $dM$ is the Riemannian volume form on the manifold of normal matrices, induced by the standard metric in $\mathbb{C}^{n^2}$ and $\mathcal Z_n>0$ is the normalization constant. Here, $\operatorname{Tr} Q(M)$ is interpreted as the sum of $Q$ evaluated at the eigenvalues of $M$. The function $Q$ is usually called the potential or the external field.

The first self-contained mathematical treatment of the random normal matrix model appears to be by Elbau and Felder \cite{ElFe} in 2005, and a lot of the present day mathematical framework has its origin in a work by Hedenmalm and Makarov \cite{HeMa2} first appearing in preprint form in 2006. The understanding of these models has been steadily increased by Ameur, Hedenmalm and Makarov, e.g. in \cite{AmHeMa2, AmHeMa3, AmHeMa}. It should be mentioned that random normal matrix models and closely related models (not having Hermitian symmetry) have been considered earlier, e.g., by Ginibre \cite{Ginibre}, Di Francesco, Gaudin, Itzykson,
and Lesage \cite{DGIL}, and Wiegmann and Zabrodin \cite{WiZa}.
To ensure integrability, one has to impose certain growth and regularity conditions on $Q$. We shall assume that $Q$ is $C^2$ on $\mathbb C$. Furthermore, we assume that
\begin{align} \label{eq:Vgrowth}
    \liminf_{|z|\to\infty} \frac{Q(z)}{\log |z|^2}>1,
\end{align}
which assures us that $\mathcal Z_n$ exists. 
The corresponding eigenvalues $z_1, \ldots, z_n\in\mathbb C$ of $M$ are distributed as
\begin{align} \label{eq:JPDF}
    d\mathbb P_n(z_1, \ldots, z_n) = \frac{1}{Z_n} \prod_{1\leq i<j\leq n} |z_i-z_j|^2 \prod_{j=1}^n e^{-n Q(z_j)} dA(z_j),
\end{align}
where $Z_n>0$ is the normalization constant \cite{ChZa}. Here, for $z= x+iy\in \mathbb C$,
\begin{align*}
    dA(z) = -\frac1{2\pi i} dz\wedge d\bar z = \frac1\pi dx \, dy
\end{align*}
denotes the standard area measure on $\mathbb C$, normalized by a factor $\pi$. There is much interest in the physics community in these eigenvalues, as they describe the locations of particles in so-called 2D Coulomb gases (with inverse temperature $2$), e.g., see \cite{Forrester}. The point process describing the eigenvalues is ``integrable" in the sense that it is entirely characterized by a function $\mathcal K_n:\mathbb C\times\mathbb C\to \mathbb C$ called the \emph{correlation kernel}. Namely, the $k$-point correlation functions can be expressed as
\begin{align} \nonumber
    \rho_{n,k}(z_1, \ldots, z_k) :=& \frac{n!}{(n-k)!} \int_{\mathbb C^{n-k}} \frac{1}{Z_n} \prod_{1\leq i<j\leq n} |z_i-z_j|^2 \prod_{j=1}^n e^{-n Q(z_j)} \, dA(z_{k+1}) \cdots dA(z_{n})\\ \label{eq:corFuncCorKer}
    =&
    \det \big(\mathcal K_n(z_i, z_j) \big)_{1\leq i,j\leq k},
\end{align}
for $k=1,\ldots,n$. The correlation kernel is not unique, and we will make the following Hermitian symmetric choice:
\begin{align} \label{eq:defSymBergmanK}
    \mathcal K_n(z,w) = e^{-\frac12n(Q(z)+Q(w))} \sum_{j=0}^{n-1} p_j(z) \overline{p_j(w)},
\end{align}
where the $p_j:\mathbb C\to\mathbb C$ are \textit{planar orthogonal polynomials}, i.e., complex polynomials of degree $j$ and positive leading coefficient  satisfying the orthogonality conditions
\begin{align*}
    \int_{\mathbb C} p_j(z) \overline{p_k(z)} e^{-n Q(z)} dA(z) = \delta_{j,k}, \qquad j, k=0,1,\ldots
\end{align*}
Such point processes are called determinantal point processes (DPPs). What is clear from this determinantal structure, and also already from the form of the joint probability density function \eqref{eq:JPDF}, is that the eigenvalues exhibit a mutual repulsion. This means that a typical configuration of eigenvalues shows less clustering than, e.g., complex numbers that are uniformly distributed on some subset of $\mathbb C$.
Under fairly general conditions on $Q$ (that we elaborate on in Section \ref{sec:kernels}), we know that the eigenvalues accumulate on a compact set $S_Q$ called the \emph{droplet} \cite{SaTo}. Explicitly, the average density of eigenvalues satisfies the following limiting behavior \cite{HeMa3}
\begin{align*}
\lim_{n\to\infty} \frac1n \rho_{n,1}(z)=
\lim_{n\to\infty} \frac1n \mathcal K_n(z,z)
= \left\{
\begin{array}{rl}
\Delta Q(z), & z\in \mathring S_Q,\\
\frac12\Delta Q(z), & z\in \partial S_Q,\\
0, & z\in S_Q^c.
\end{array}\right.
\end{align*}
(For the limit on $\partial S_Q$, see \cite{HeWe}.) Here $\Delta$ denotes the quarter Laplacian:
\begin{align} \label{eq:quarterLaplacian}
    \Delta = \partial_z \bar\partial_z
    = \frac12\left(\partial_x-i\partial_y\right)
    \frac12\left(\partial_x+i\partial_y\right)
    = \frac14 \left(\partial_x^2+\partial_y^2\right).
\end{align}
Our focus will be the behavior of linear statistics of the form
\begin{align} \label{eq:linStatf}
\mathfrak X_n(f) = f(z_1)+\ldots+f(z_n),
\end{align}
where $z_1, \ldots, z_n$ are the eigenvalues and $f$ is a test function. One is frequently interested in the variance (and higher order cumulants) of these linear statistics as $n\to\infty$. For smooth test functions, these linear statistics are well understood in the large $n$ limit \cite{AmHeMa3, AmHeMa}.   Our focus will be the qualitatively different case of \emph{counting statistics} where the test function is an indicator function
\[
f(z) = \mathfrak{1}_A(z)
= \begin{cases}
    1, & z\in A,\\
    0, & z\in A^c,
\end{cases}
\]
and $A$ is a given subset of $\mathbb C$. We shall henceforth use the notation
\begin{align*}
    N_A^{(n)} = \mathfrak X_n(\mathfrak 1_A)
\end{align*}
for counting statistics. Note that $N_A^{(n)}$ counts the number of eigenvalues in $A$. Using the weak convergence (e.g., see \cite{Berman}) of $\frac1n \rho_{n,1}$, one infers that
\[
\lim_{n\to\infty}\frac1n N_A^{(n)} = \int_A \Delta Q(z) dA(z),
\]
for any measurable set $A$. 
The corresponding variance of $N_A^{(n)}$ is called the \emph{number variance}, and it will be our main object of interest. Counting statistics of random normal matrices (or explicit models in this class) have attracted much attention recently, see \cite{LGCCKMS, LMS,LMS2, FeLa, Charlier1, ChLe, AkByEb, AmChCrLe, AmChCrLe2, Lin, LeMaOC, Noda}.

There is a direct relation between the variance of the linear statistic defined in \eqref{eq:linStatf}, and the correlation kernel, namely
\begin{align} \label{eq:numberVarCorKerInt}
\operatorname{Var} \mathfrak{X}_n(f)
= \frac12 \int_{\mathbb C^2} |f(z)-f(w)|^2 |\mathcal K_n(z,w)|^2 dA(z) dA(w). 
\end{align}
In our case $f=\mathfrak{1}_A$ and the number variance can then be alternatively expressed as
\begin{align} \label{eq:numberVarCorKerIntA}
\operatorname{Var} N_A^{(n)}
&= \frac12 \int_{\mathbb C^2} |\mathfrak{1}_A(z)-\mathfrak{1}_A(w)| |\mathcal K_n(z,w)|^2 dA(z) dA(w)\\ \nonumber
&= \int_A \int_{A^c} |\mathcal K_n(z,w)|^2 dA(z) dA(w). 
\end{align}
One of the first results on the number variance for the eigenvalues of a random normal matrix model, seemingly lesser known, can be found in a 2004 paper by Rider \cite{Rider}, where the limiting behavior of the number variance for the Ginibre ensemble $Q(z)=|z|^2$ is found in the case where $A$ is either an annulus or an angular segment. (Technically, the matrices sampled from the Ginibre ensemble are not assumed to be normal, but imposing this condition or not turns out to be irrelevant for the distribution of eigenvalues).

In \cite{AkByEb} Akemann, Byun and Ebke, likely inspired by \cite{LGCCKMS, LMS, LMS2, FeLa, Charlier1, ChLe}, derived a universality result for the counting statistics when $Q$ is radially symmetric, i.e., $Q(z)=Q(|z|)$. In these papers the set $A$ is a concentric disk and the counting statistics are called disk counting statistics. Notably, \cite{Charlier1} also considered the case where $Q$ has a logarithmic singularity. Under certain general conditions and normalizations on $Q$ (see Definition 1.1 in \cite{AkByEb}), in particular to ensure that $S_Q$ is the closed unit disc $\overline{\mathbb D}$ (rather than a union of a point or disc and annuli), they showed that when $A$ is a concentric disc $B(0,a)$ with $0<a<1$, then one has
\begin{align*}
\lim_{n\to\infty} \frac1{\sqrt n} \operatorname{Var }N_A^{(n)} = \frac{a}{\sqrt{\pi}} \sqrt{\Delta Q(a)}. 
\end{align*}
For this result, radial symmetry is necessary for both the potentials $Q$ and the sets $A$. In \cite{LMS,LMS2,Lin,LeMaOC} the only random normal matrix model considered was $Q(z)=|z|^2$, the Ginibre ensemble. The result in \cite{Lin, LeMaOC} holds for general sets $A$. Namely, if $A$ is a Borel set, they showed that
\begin{align*}
\lim_{n\to\infty} \frac1{\sqrt n} \operatorname{Var }N_A^{(n)} = \frac{1}{2\pi\sqrt\pi} \mathcal H^1(\partial_* A), 
\end{align*}
where $\partial_* A$ is the measure theoretic boundary of $A$, and $\mathcal H^1(\partial_* A)$ is the perimeter of $A$. The measure theoretic boundary is defined explicitly by
\begin{align*}
    \partial_* A = \left(\left\{z\in\mathbb C : \lim_{r\downarrow 0} \frac{m_2(A\cap B(z,r))}{\pi r^2}=1\right\}\cup \left\{z\in\mathbb C : \lim_{r\downarrow 0} \frac{m_2(A\cap B(z,r))}{\pi r^2}=0\right\}\right)^c
\end{align*}
where $m_2$ denotes the two-dimensional Lebesgue measure.
When $A$ has a $C^1$ boundary, the measure theoretic boundary is the same as the topological boundary $\partial A$, and $\mathcal H^1(\partial_* A)$ corresponds to the length of the topological boundary (one may also replace $\partial_* A$ by the reduced boundary here and in the formulae below, as the difference has measure $0$). Note that the limit exists if and only if $A$ is a Caccioppoli set, i.e., if and only if it has a finite perimeter. Heuristically, this result makes sense: in the bulk $\mathring S_Q$ the kernel $\mathcal K_n(z,w)$ (for any generic $Q$) behaves like a Gaussian that is sharply peaked along the diagonal $z=w$. Combining this with the second line of \eqref{eq:numberVarCorKerIntA} provides the insight that the boundary encompasses the dominant behavior. A similar limiting behavior was seen in the number variance of stationary point processes (to which the Ginibre ensemble asymptotically belongs)\cite{SoWeYa, KrYo}. 
Very recently, it was proved in \cite{AkDuMo2} for a general class of potentials $Q$ that the number variance is of order $|\partial A| \sqrt n$ for all convex sets $A\Subset \mathring S_Q$ with $C^2$ boundary. This general class $Q$ requires that $Q$ is $C^2$, real analytic on an open neighborhood of $S_Q$, and $\Delta Q>0$ on a compact set containing $A$. As mentioned in \cite{AkDuMo}, this implies that $N_A^{(n)}$ satisfies a central limit theorem after proper rescaling due to Soshnikov's theorem \cite{Soshnikov}. Moreover, when $\Delta Q$ is constant on the droplet, it was shown that
\begin{align*}
\operatorname{Var }N_A^{(n)} = \frac{\sqrt n}{2\pi\sqrt\pi} |\partial A| \sqrt{\Delta Q|_{S_Q}}
+ \mathcal O(1),
\end{align*}
uniformly for $n=1,2,\ldots$ and uniformly for convex sets $A\subset K$ with $C^2$ boundary, where $K\subset \mathring S_Q$ is a given compact set.  

In this paper we go a considerable step further, and prove a universal limiting behavior for the number variance for general potentials $Q$ and general sets $A$. 
We prove the following general result. 

\begin{theorem} \label{thm:main}
Consider a random normal matrix model with a potential $Q$ that is $C^2$, real analytic on $\mathring S_Q$ and $\Delta Q>0$ on $S_Q$. 
For any Borel set $A\Subset \mathring S_Q$, we have 
\begin{align} \label{eq:mainThm}
\lim_{n\to\infty} \frac1{\sqrt n}\operatorname{Var }N_A^{(n)} = \frac1{2\pi\sqrt\pi}\int_{\partial_* A} \sqrt{\Delta Q(z)} d\mathcal H^1(z),
\end{align}
where $\partial_* A$ is the measure theoretic boundary of $A$.
\end{theorem}
\noindent We remind the reader that $\Delta$ denotes the quarter Laplacian as defined in \eqref{eq:quarterLaplacian}. Note that when $A$ has a $C^1$ boundary, the measure theoretic boundary is the same as the topological boundary $\partial A$, and $d\mathcal H^1(z)$ is simply the usual arc length differential. A limitation of our method is that we do not have a precise understanding of the error in \eqref{eq:mainThm}. We suspect, however, that the limit holds with an error of order $\mathcal O(1/\sqrt n)$, or perhaps $\mathcal O(\log(n)/\sqrt n)$. This suspicion comes from \cite{AkDuMo}, where requiring the boundary of $A$ to be $C^1$ rather than $C^2$ seems to produce a term of order $\mathcal O(\log(n)/\sqrt n)$, although this could be an artifact of the method. Nevertheless, we do have a good understanding of the error in terms of an upper bound, see Corollary \ref{eq:thmMainUpperBound}. 


Another interesting phenomenon is the number variance on a microscopic dilation of the droplet, i.e., where $A$ is either microscopically smaller or microscopically larger than the droplet in some specified way. This was first considered by Lacroix-A-Chez-Toine, Majumdar and Schehr for the complex Ginibre ensemble, who used it to describe non-interacting Fermions in a rotating trap \cite{LMS}. This result was later also verified in \cite{FeLa}. When $Q(z)=|z|^2$ and $A=A_n(\delta)=\{z\in\mathbb C : |z|\leq 1+\frac{\delta}{\sqrt{2n\Delta Q(1)}}\}$ for some fixed $\delta\in\mathbb R$, they showed that
\begin{align*}
        \lim_{n\to\infty} \frac1{\sqrt n}\operatorname{Var} N_{A_n(\delta)}^{(n)}
        &= \frac{f(\delta)}{\sqrt\pi},
        & f(\delta)=\sqrt{2\pi}\int_{\delta}^\infty \frac{\operatorname{erfc}(t) \operatorname{erfc}(-t)}{4} dt
    \end{align*}
Afterwards, Akemann, Byun and Ebke generalized their result to the rotationally symmetric setting \cite{AkByEb}. In \cite{AkDuMo} a similar result was shown for the complex elliptic Ginibre ensemble, corresponding to $Q(z)=|z|^2-\tau\Re(z^2)$ (or a rescaling thereof), which does not exhibit radial symmetry. We go a considerable step further in this paper and consider the analogous result for general potentials (not necessarily rotationally symmetric). We restrict our attention to the case that $S_Q$ is connected and has a smooth boundary. We additionally assume that $S_Q$ equals the so-called \textit{predroplet}. We postpone the somewhat involved explanation of this commonly imposed condition to Section \ref{sec:kernels} below. 
For any $\delta\in\mathbb R$, we define the following tubular neighborhood of $\partial S_Q$ given by
\begin{align*}
    S_{Q,n}^\delta = \{h_n(z,t) : z\in \partial S_Q, \, |t|<|\delta|\},
\end{align*}
where 
\begin{align} \label{eq:defhn}
h_n(z,t)=z+\frac1{\sqrt{2n\Delta Q(z)}}\vec n(z) t.
\end{align}
Here $\vec n(z)$ denotes the outward unit normal vector on $\partial S_Q$ at $z$. 
By an application of the inverse function theorem, $h_n$ is a diffeomorphism from $\partial S_Q\times (-|\delta|,|\delta|)$ to $S_{Q,n}^\delta$ for large enough $n$. Now consider our counting statistic $N_A^{(n)}$ for
\begin{align} \label{eq:defAn}
    A = A_n(\delta) = 
    \begin{cases}
        S_Q \cup S_{Q,n}^{\delta}, & \delta\geq 0,\\
        S_Q \setminus S_{Q,n}^{\delta}, & \delta<0.
    \end{cases}
\end{align}
In other words, when $\delta>0$ the set $A$ is microscopically bigger than the droplet while, when $\delta<0$, the set $A$ is microscopically smaller than the droplet. In this setup we have the following theorem for the corresponding number variance. We remind the reader that the harmonic measure at $\infty$ corresponding to $S_Q^c$ is given by
\begin{align*}
    d\omega_{S_Q^c}^\infty(z) = |\phi'(z)| d\mathcal H^1(z),
\end{align*}
where $\phi$ is the conformal map from $S_Q^c$ to $\overline{\mathbb D}^c$ satisfying $\phi(\infty)=\infty$ and, though it is irrelevant for the definition of the harmonic measure, here and in the rest of the paper we fix $\phi'(\infty)>0$ (i.e., $\phi$ is assumed to be orthostatic). 

\begin{theorem} \label{thm:dilationNumberVar}
Consider a random normal matrix model with a potential $Q$ that is $C^2$ on $\mathbb C$, and real analytic and strictly subharmonic on a neighborhood of $S_Q$. Assume that $S_Q$ is simply connected, has a smooth boundary, and that it equals the predroplet.
    Then 
    \begin{align*}
        \lim_{n\to\infty} \frac1{\sqrt n}\operatorname{Var} N_{A_n(\delta)}^{(n)}
        = \frac{f(\delta)}{2\pi\sqrt\pi}  \int_{\partial S_Q} \sqrt{\Delta Q(z)} \,  d\omega_{S_Q^c}^{\infty}(z)
    \end{align*}
    uniformly for $\delta\in\mathbb R$ in compact sets, where $\omega_{S_Q^c}^{\infty}$ is the harmonic measure at $\infty$, and
    \begin{align*}
        f(\delta) = \sqrt{2\pi}\int_{\delta}^\infty \frac{\operatorname{erfc}(t) \operatorname{erfc}(-t)}{4} dt.
    \end{align*}
\end{theorem}
Comparing with Theorem 1.2 in \cite{AkByEb} in the rotationally symmetric case, they have $\phi(z)=z$ and thus $d\omega_{S_Q^c}^\infty(z)=d\mathcal H^1(z)$, and this is indeed in agreement. As mentioned in \cite{AkByEb}, we have $f(-\infty)=1, f(0)=\frac12$ and $f(\infty)=0$. However, contrary to the rotationally symmetric setting, for general potentials, Theorem \ref{thm:dilationNumberVar} cannot directly be considered an interpolation between the droplet (Theorem \ref{thm:main}) and the region around the edge, due to the different measures that are integrated over. However, one could interpret the limit $\delta\to-\infty$ as recovering the bulk case, where one integrates over the boundary of $A=S_Q$. 
Notice that one may pull back the harmonic measure and alternatively write
\begin{align*}
    \int_{\partial S_Q} \sqrt{\Delta Q(z)} d\omega_{S_Q^c}^\infty(z) = \int_{\partial\mathbb D} \sqrt{\Delta Q(\phi^{-1}(\zeta))} d\mathcal H^1(\zeta).
\end{align*}
Furthermore, this particular term is the solution to the Dirichlet problem on $S_Q^c$ with boundary values $\sqrt{\Delta Q(z)}$ on $\partial S_Q$, evaluated at $z=\infty$. When one considers the matrices in \eqref{eq:defJPDFM} to be Hermitian rather than normal, one has the Gaussian unitary ensemble (GUE) (with general potential). A comparison of Theorem \ref{thm:main} and Theorem \ref{thm:dilationNumberVar} with their one-dimensional counterparts is in order. When $A$ is strictly inside the limiting support of eigenvalues of the GUE one finds that
\begin{align*}
    \operatorname{Var} N_A^{(n)} = \frac1{\pi^2} \log n + \mathcal O(1)
\end{align*}
as $n\to\infty$, where the implied constant depends on $A$ \cite{CoLe}. Actually, it can be extracted from \cite{Charlier0} that the same large $n$ behavior holds in the case of general potentials, in particular with the leading order behavior independent of the potential. On the other hand, going back to the GUE, when one takes a microscopic dilation $A_n(\delta)=[-2+\delta/n,2-\delta/n]$ one finds a limiting bounded function $\sigma$ \cite{Soshnikov} such that
\begin{align*}
    \lim_{n\to\infty} \operatorname{Var} N_{A_n(\delta)}^{(n)} = \sigma(\delta)^2.
\end{align*}
Thus, in sharp contrast with the model of random normal matrices, the edge behavior is of a different order, $\mathcal O(1)$ instead of $\mathcal O(\log n)$ \cite{Johansson}. This is called number variance saturation \cite{Johansson} (i.e., there is an upper bound to how much the eigenvalues can fluctuate in an interval of microscopic size), and it does not happen in the model of random normal matrices.

Theorem \ref{thm:dilationNumberVar} does not touch on certain more exotic settings. For example, one may consider the case that the droplet has a hard edge, i.e., the eigenvalues are excluded in a certain region. This has been considered in the rotationally symmetric case in \cite{AmChCrLe, AmChCrLe2}. Notably, \cite{AmChCrLe2} also considers the case of disconnected droplets. Based on the very recent paper \cite{CrWe}, one would expect the limiting variance in Theorem \ref{thm:dilationNumberVar} to be a sum with harmonic measures corresponding to the exterior of each connected component. 

As in Theorem \ref{thm:main}, the dominant contributions to the variance come from the vicinity of the boundary of $A$, which in the case of Theorem \ref{thm:dilationNumberVar} corresponds to the droplet boundary, or \emph{edge}, $\partial S_Q$. 
The edge behavior of the kernel has been investigated in generality by Hedenmalm and Wennman \cite{HeWe} and Ameur and Cronvall \cite{AmCr}, and was for the explicit models of the Ginibre and elliptic Ginibre ensemble already investigated by Lee and Riser \cite{LeRi}. 
Key to our proof is the following strengthening of a result by Hedenmalm and Wennman  \cite{HeWe}(Corollary 1.7). First, our result implies that their result holds uniformly under certain conditions on $z$ and $w$ that allows them to have a distance of higher order than $1/\sqrt n$ from the boundary. Second, contrary to \cite{HeWe} we allow for the possibility that $z_0\neq w_0$ below. The $\log$ is defined with argument in $(-\pi,\pi]$ in what follows, and we define the complementary error function as
\begin{align*}
    \operatorname{erfc}(\zeta) = \frac{2}{\sqrt\pi} \int_{\zeta}^\infty e^{-z^2} dz.
\end{align*}

\begin{lemma} \label{lem:keyLemma}
    Consider a random normal matrix model with a potential $Q$ that is $C^2$ on $\mathbb C$, and real analytic and strictly subharmonic on a neighborhood of $S_Q$. Assume that $S_Q$ is simply connected, has a smooth boundary, and that it equals the predroplet. 
    Let $z_0,w_0\in \partial S_Q$ and denote by $\vec n(z_0)$ and $\vec n(w_0)$ the outward unit normal vectors on $\partial S_Q$ at $z_0$ and $w_0$. Then we have as $n\to\infty$ that
    \begin{multline} \label{eq:HeWeGen}
        \frac1{n\sqrt{\Delta Q(z_0)\Delta Q(w_0)}}\left|\mathcal K_n\left(z_0+\frac{\vec n(z_0) \xi}{\sqrt{n \Delta Q(z_0)}},w_0+\frac{\vec n(w_0)\eta}{\sqrt{n \Delta Q(w_0)}}\right)\right| = \\
        (\frac12+\mathcal O(\frac{\log^3 n}{\sqrt n})) \exp\left(-\frac12|\xi-\eta|^2-n\Delta Q(z_0)\frac{(\log(\phi(z_0)\overline{\phi(w_0)}))^2}{2|\phi'(z_0)|^2}\right)\\\left|\operatorname{erfc}\left(\frac{\xi+\overline\eta}{\sqrt 2}+ \sqrt{n\Delta Q(z_0)}\frac{\log(\phi(z_0)\overline{\phi(w_0)})}{\sqrt 2|\phi'(z_0)|}\right)\right|
\end{multline} 
        uniformly for $|z_0-w_0|=\mathcal{O}(\sqrt\frac{\log n}{n})$ and $\xi,\eta=\mathcal O(\sqrt{\log n})$, where $\phi$ is the conformal map from $S_Q^c$ to $\overline{\mathbb D}^c$ such that $\phi(\infty)=\infty$ and $\phi'(\infty)>0$.
\end{lemma}


More symmetric versions, treating $z_0$ and $w_0$ equally, are also possible and can be derived with our method, but we shall prefer the version as stated for the application we have in mind. 
Interestingly, when $|z_0-w_0|\to 0$ but $\sqrt n |z_0-w_0|\to \infty$, and $\xi,\eta= o(\sqrt{\log n})$, the asymptotic behavior of the complementary error function yields a term
\begin{align*}
\frac1{\log(\phi(z_0)\overline{\phi(w_0)})} &= \frac1{\phi(z_0)\overline{\phi(w_0)}-1}+\sum_{k=0}^\infty \frac{B_{k+1}}{(k+1)!}(\phi(z_0)\overline{\phi(w_0)}-1)^k\\
    &= \frac{1}{\phi(z_0)\overline{\phi(w_0)}-1}
    \left(1+\mathcal O(z_0-w_0)\right)
\end{align*}
and we get the Szeg\H{o} kernel result from \cite{AmCr} (Theorem 1.3 and Corollary 1.4) which holds under the different condition $|\phi(z_0)\overline{\phi(w_0)}-1|\geq \epsilon$ for some fixed $\epsilon>0$ (and $B_k$ are the Bernoulli numbers of the second kind). 
This Szeg\H{o} kernel associated to the curve $\partial S_Q$ is the reproducing kernel on the Hardy space $H_0^2(\overline{S_Q^c})$ of holomorphic functions that vanish at infinity and are square integrable with respect to arc length over $\partial S_Q$. It is explicitly given by
\begin{align*}
    \mathcal S(z,w) = \frac1{2\pi}\frac{\sqrt{\phi'(z)}\overline{\sqrt{\phi'(w)}}}{\phi(z)\overline{\phi(w)}-1},
\end{align*}
with $\phi$ being the conformal map from $S_Q^c$ to $\overline{\mathbb D}^c$ such that $\phi(\infty)=\infty$ and $\phi'(\infty)>0$.
To understand the contribution of the remaining region $|z-w|^{-1}=\mathcal O(\sqrt\frac{n}{\log n})$ near the droplet boundary, we have the following lemma.

\begin{lemma} \label{lem:keyLemma2}
Consider a random normal matrix model with a potential $Q$ that is $C^2$ on $\mathbb C$, and real analytic and strictly subharmonic on a neighborhood of $S_Q$. Assume that $S_Q$ is simply connected, has a smooth boundary, and that it equals the predroplet. Then as $n\to\infty$
\begin{multline*}
\frac1{\sqrt n}\left|\mathcal K_n\left(z_0+\frac{\vec n(z_0) \xi}{\sqrt{n \Delta Q(z_0)}},w_0+\frac{\vec n(w_0)\eta}{\sqrt{n \Delta Q(w_0)}}\right)\right|\\
\leq C_Q \left|\mathcal S\left(z_0+\frac{\vec n(z_0) \xi}{\sqrt{n \Delta Q(z_0)}},w_0+\frac{\vec n(w_0)\eta}{\sqrt{n \Delta Q(w_0)}}\right)\right| e^{-(\Re \xi)^2} e^{-(\Re\eta)^2},
\end{multline*}
uniformly for $z_0,w_0\in \partial S_Q$ and $\xi,\eta=\mathcal O(\sqrt{\log n})$, where $C_Q>0$ is a constant that depends only on $Q$, and $\mathcal S$ is the Szeg\H{o} kernel associated with $\partial S_Q$.
\end{lemma}

When $C_Q$ is chosen optimally (or close to the infimum of such constants), we expect this inequality to be essentially sharp for large $n$ near the edge for $z$ and $w$ not too close to each other, as can be seen in the complex Ginibre ensemble and in Theorem I.1 in \cite{AkDuMo} for the complex elliptic Ginibre ensemble.
We mention that, while Lemma \ref{lem:keyLemma2} is formally correct for $z-w=\mathcal O(\sqrt\frac{\log n}{n})$, it is Lemma \ref{lem:keyLemma} that provides detailed information in that situation. The reason for including Lemma \ref{lem:keyLemma} and Lemma \ref{lem:keyLemma2} in the introduction is that we consider them to be of independent interest for researchers in random normal matrices and related fields. We believe that the conditions for $z_0-w_0$ and $\xi, \eta$ in Lemma \ref{lem:keyLemma} and Lemma \ref{lem:keyLemma2} are sufficient in scenarios where one needs to integrate over an integrand involving the kernel. However, based on implicit results in \cite{Hedenmalm,HeWe2} it seems that Lemma \ref{lem:keyLemma2} holds under the weaker assumption $\xi,\eta=\mathcal O(n^{1/4})$, while Lemma \ref{lem:keyLemma} likely holds under the weaker conditions $z_0-w_0=\mathcal O(n^{-1/4-\epsilon})$ and $\xi,\eta=\mathcal O(n^{1/4-\epsilon})$ for any fixed $\epsilon>0$, although an extra error term $\mathcal O(n^{-2\epsilon}\log n)$ has to be added to the factor $1+\mathcal O(\frac{\log^3n}{\sqrt n})$. Lastly, we mention that a result similar to Lemma~\ref{lem:keyLemma} for the case $z_0=w_0$ is simultaneously and independently proved in \cite{Christophe}, where the form of the subleading term is obtained explicitly.

The paper is organized as follows. In Section \ref{sec:kernels} we introduce the reader to some information and estimates on the correlation kernel which are necessary to prove Theorem \ref{thm:main} and Theorem \ref{thm:dilationNumberVar}. Our approach to prove Theorem \ref{thm:main} involves approximating $\mathfrak{1}_A$ by functions of bounded variation. This method was used in \cite{Lin,LeMaOC} but the kernel, contrary to the models considered in \cite{Lin,LeMaOC} are in general not asymptotically radial and our accomplishment is to overcome this difficulty. Actually, we prove Theorem \ref{thm:main} with a weighted bounded variation norm representation result, Proposition \ref{teo:BV}, in the spirit of \cite{Davila, BoBrMi}, for arbitrary $f\in L^1$. In Section \ref{sec:BV} we provide some preliminaries on functions of bounded variation. To prove Theorem \ref{thm:dilationNumberVar}, we need Lemma \ref{lem:keyLemma} and Lemma \ref{lem:keyLemma2}. We prove these in Section \ref{sec:keyLemmas}. Finally, the proofs of Theorem \ref{thm:main} and Theorem \ref{thm:dilationNumberVar} are in, respectively, Section \ref{sec:proofMain} and Section \ref{sec:proofDilation}.


\section{Correlation kernels and Bergman kernels}
\label{sec:kernels}

In this section, we present some preliminaries and estimates on the correlation kernel that are needed to prove the main results.
There is an alternative interpretation of the correlation kernel in the case of random normal matrices. Namely, the correlation kernel without the weight factors, that is
\begin{align*}
\boldsymbol k_n(z,w)= \sum_{j=0}^{n-1} p_j(z) \overline{p_j(w)},
\end{align*}
is known as the (polynomial) Bergman kernel on the Hilbert space of complex polynomials $p$ of degree $<n$ with respect to the norm \begin{align*}
    \int_{\mathbb C} |p(z)|^2 e^{-n Q(z)} dA(z).
\end{align*}
Such a kernel satisfies the reproducing property, i.e., for any complex polynomial $p$ of degree $<n$ we have that
\begin{align*}
    p(z) = \int_{\mathbb C} \boldsymbol k_n(z,w) p(w) e^{-n Q(w)} dA(w).
\end{align*}
We know how to approximate such a kernel inside the droplet for $z$ and $w$ close to each other. This is in the first place due to Berman \cite[Theorem 3.8]{Berman} and in the second place due to Ameur, Hedenmalm and  Makarov \cite[Theorem 2.3]{AmHeMa2}, who refined Berman's results to the random normal matrix setting. Fix any compact set $K \subset \mathring S_Q$, and assume that $\Delta Q>0$ on $K$. Then there exists an $\varepsilon=\varepsilon_Q^K>0$ and a constant $C_Q^K>0$ depending only on $Q$ and $K$ such that uniformly for $|z-w|<\varepsilon$
\begin{align} \label{eq:firstAppBergmanK}
        |\mathcal K_n(z,w)-\mathcal K_n^1(z,w)| \leq \frac{C_Q^K}{n}
\end{align}
for $n$ big enough, where  $\mathcal K_n^1(z,w)$ is the weighted first order approximation of the Bergman kernel. Higher order approximating Bergman kernels can be constructed as well, leading to an error of arbitrarily high negative power of $n$, but we shall only need the first order approximation. It is defined as
\begin{align} \label{eq:BergmanApprox1stO}
        \mathcal K_n^1(z,w) = \big(n B_0(z,\overline{w})+B_1(z, \overline{w})\big) e^{n (Q(z, \overline{w})-\frac12Q(z)-\frac12Q(w))}, 
    \end{align}
where 
\begin{align*}
        B_0(z,w) = \partial_z\partial_w Q(z,w),
        \qquad B_1(z,w)= \frac12 \partial_z\partial_w \log \partial_z\partial_w Q(z,w).
    \end{align*}
The function $Q(z,w)$ is the polarization of $Q(z)$, i.e., the unique function in a neighborhood of $w=\overline z$ that is analytic in both its variables and such that $Q(z,\overline z)=Q(z)$. It is guaranteed to exist on any compact set for $|z-\overline w|<\varepsilon$ if we pick $\varepsilon$ small enough. These facts and \eqref{eq:firstAppBergmanK} can be utilized to prove several convenient estimates for the correlation kernel \cite{AmHeMa2}. Under the conditions for $Q$ as in Theorem \ref{thm:main} we have that 
\begin{align} \label{eq:KnzzK}
\frac1n \mathcal K_n(z,z) = \Delta Q(z) + \mathcal O(1/n)
\end{align}
as $n\to\infty$, uniformly for $z\in K$ (e.g., see \cite{HeMa3}). Furthermore, we have the estimate
\begin{align} \label{eq:Knzw}
|\mathcal K_n(z,w)| \leq C n e^{-\epsilon \sqrt n \min(d_K,|z-w|)} e^{-\frac12n (Q(z)-\check Q(z))},
\end{align}
uniformly for $w\in K$ and $z\in\mathbb C$, for some constants $C, \epsilon>0$ that depend on $Q$ and $K$, and $d_K$ is the distance between $K$ and $S_Q^c$ \cite[Theorem 2.12]{AmHeMa2}. Here $\check Q$ denotes the obstacle function (sometimes called the equilibrium potential), i.e., the maximal subharmonic function that satisfies both $\check Q\leq Q$ and $\check Q(z)=\log|z|^2+\mathcal O(1)$ as $|z|\to \infty$.
The \textit{predroplet} is defined as the coincidence set $\check Q=Q$. Under the assumption that $Q$ is $C^2$ (or even $C^{1,1}$) the predroplet equals $S_Q\cup \mathcal N$ where $\mathcal N$ is a set of measure $0$, with respect to a certain measure $\sigma_Q$ that we define in a moment. Namely, the obstacle function has a potential theoretic interpretation. Defining the logarithmic potential
\begin{align*}
    U_Q(z) = \int_{\mathbb C} \log \frac1{|z-w|} d\sigma_Q(w),
\end{align*}
we have that
\begin{align*}
    \check Q+2 U_Q = c_Q,
\end{align*}
where $c_Q$ is a Robin-type constant and $\sigma_Q$ is the \emph{equilibirum measure}, given explicitly by $d\sigma_Q(z)=\Delta\check Q(z) dA(z)= \Delta Q(z) \mathfrak{1}_{S_Q}(z) dA(z)$. It solves a potential theoretic equilibrium problem that is obtained by a standard heuristic continuum limit argument. Namely, rewriting \eqref{eq:JPDF} as
\begin{align*}
    \frac1{Z_n} \exp\left(n^2\left( \frac1{n^2}\sum_{j\neq k} \log|z_j-z_k|-\frac1{n} \sum_{j=1}^n Q(z_j)\right)\right),
\end{align*}
and replacing the sums $\frac1n \sum_{j=1}^n$ by integrations over some measure, one asks to minimize the (energy) functional
\begin{align*}
    I_Q(\mu) = \iint_{\mathbb C^2} \log \frac1{|z-w|} d\mu(z) d\mu(w)
    + \int_{\mathbb C} Q(z) d\mu(z),
\end{align*}
over all compactly supported
Borel probability measures on $\mathbb C$. 
It turns out that $\mu=\sigma_Q$ is the minimizer\cite{SaTo}.
One immediate consequence of \eqref{eq:Knzw} is that
\begin{align} \label{eq:KnzzC}
\mathcal K_n(z,z) \leq C n
\end{align}  
uniformly for $z\in\mathbb C$. One object, investigated in \cite{AmHeMa2}, that we will also use, is the Berezin measure rooted at $w\in\mathbb C$. It is defined as the measure $B_n^{(w)}(z) dA(z)$, where $B_n^{(w)}$ is the Berezin kernel rooted at $w$ given explicitly by
\begin{align} \label{eq:defBerezin}
B_n^{(w)}(z) = \frac{|\mathcal K_n(z,w)|^2}{\mathcal K_n(w,w)}.
\end{align}
By the reproducing property of the kernel, it defines a probability measure. It was shown in \cite{AmHeMa2} that the normalized Berezin measure $\hat B_n^{(w)}(z) dA(z)$, defined via
\begin{align} \label{eq:defNomBerezin}
\hat B_n^{(w)}(z) = \frac1{n\Delta Q(w)} B_n^{(w)}\left(w+\frac{z}{\sqrt{n \Delta Q(w)}}\right),
\end{align}
converges in distribution to $e^{-|z|^2} dA(z)$. This implies in particular that $B_n^{(w)}$ converges to the Dirac measure $\delta_w$ in the weak-star sense. Furthermore, we have the behavior
\begin{align} \nonumber
B_n^{(w)}(z) &=  \frac{n^2 |B_0(w, \overline z)|^2+\mathcal O(n)}{n \Delta Q(w)+\mathcal O(1)} e^{-n \Delta Q(w)|z-w|^2+\mathcal O(n|z-w|^3)}+\mathcal O(n^{-2})\\ \label{eq:behavBerezin}
&=(1+\mathcal O(n |z-w|^3)+\mathcal O(n^{-1})) n \Delta Q(w) e^{-n \Delta Q(w)|z-w|^2}+\mathcal O(n^{-2}),
\end{align}
uniformly for $w\in K$ and $|z-w|=\mathcal O(\frac{\log n}{\sqrt n})$, which is a slightly more refined version of (7.5) in \cite{AmHeMa2}. 

\section{Functions of bounded variation} \label{sec:BV}

On $\mathbb C$ the space of functions of bounded variation \cite{AmFuPa} is defined as
\begin{align*}
BV(\mathbb C) = \{f\in L^1(\mathbb C) : [f]_{BV}<\infty\},
\end{align*}
where $[f]_{BV}$ denotes the total variation of $f$:
\begin{align*}
[f]_{BV} = \sup\left\{\int_{\mathbb C} f(z) \operatorname{div}\phi(z)\, dA(z) : \phi\in C_c^\infty(\mathbb C,\mathbb R^2)\text{ with }\lVert \phi\rVert_{L^\infty}\leq 1\right\}.
\end{align*}
The space $BV(\mathbb C)$ is usually endowed with the norm
\begin{align*}
\lVert f\rVert_{BV} = \lVert f\rVert_{L^1}+[f]_{BV},
\end{align*}
and this turns it into a Banach space. 
It is a simple exercise to verify that
\begin{align*}
\left|\int_{\mathbb C} f(z) \operatorname{div} \phi(z) dA(z)\right|
\leq \lVert \nabla f\rVert_{L^1}
\end{align*}
whenever $f\in C_c^\infty(\mathbb C)$. 
Interpreting the gradient in the distributional sense, the expression on the right-hand side is the total variation for functions in $W^{1,1}$, but the Sobolev space $W^{1,1}$ is a proper subspace of $BV(\mathbb C)$. However, any function $f\in BV(\mathbb C)$ can be approximated by functions in $C_c^\infty(\mathbb C)$ in the following way. There exists a sequence $(f_j)_{j=1}^\infty$ in $C_c^\infty(\mathbb C)$ such that
\begin{align} \label{eq:ffjBVseq}
\lim_{j\to\infty} \lVert f-f_j\rVert_{L^1} = 0\text{ and } \lim_{j\to\infty}\int_{\mathbb C} |\nabla f_j(z)|\, dA(z) = [f]_{BV}. 
\end{align}
In general, it is not possible to approximate a (non-trivial) BV function $f$ by a sequence $f_j\in C_c^\infty(\mathbb C)$ in $W^{1,1}$ norm, and \eqref{eq:ffjBVseq} is essentially the best way to approximate the distributional derivative of $f$ by such a sequence.
Similar conclusions are valid when $dA(z)$ is replaced by $\omega(z) dA(z)$, when $\omega$ is some positive $C^1$ weight \cite{Baldi}. That is, if
\begin{align*}
[f]_{BV}^\omega := \sup\{\int_{\mathbb C} f(z) (\operatorname{div} \phi(z)) \omega(z) dA(z) : \phi\in C_c^\infty(\mathbb C,\mathbb R^2)\text{ with }\lVert \phi\rVert_{L^\infty}\leq 1\}<\infty,
\end{align*}
then there exists a sequence $(f_j)_{j=1}^\infty$ in $C_c^\infty(\mathbb C)$ such that
\begin{align} \label{eq:fjtofBV}
\lim_{j\to\infty} \lVert (f-f_j) \omega\rVert_{L^1} = 0\text{ and } \lim_{j\to\infty}\int_{\mathbb C} |\nabla f_j(z)| \omega(z) dA(z) = [f]_{BV}^\omega. 
\end{align}
One way of seeing this is to write $(\nabla f) \omega=\nabla (f \omega)-f \nabla \omega$ and use the earlier sequence for the BV function $f \omega$ instead of $f$.
One particular class of functions of bounded variation in which we are interested consists of indicator functions $\mathfrak{1}_A$ corresponding to some set $A\subset \mathbb C$. It follows from the structure theorem of De Giorgi \cite{DeGiorgi} that $[\mathfrak{1}_A]_{BV}$ equals the perimeter of $A$, i.e., the integral over the measure theoretic boundary $\partial_*A$ of $A$ with respect to the one-dimensional Hausdorff measure. We write this as
\begin{align*}
[\mathfrak{1}_A]_{BV} = \int_{\partial_*A} d\mathcal H^1(z) = \mathcal H^1(\partial_*A).
\end{align*}
When $A$ has a $C^1$ boundary, the measure theoretic boundary equals the topological boundary, and $d\mathcal H^1(z)$ is the same as the arc-length differential. In the case of a weight this formula changes to (e.g., use \cite[Theorem 3.84]{AmFuPa} for $\nabla(\omega \mathfrak{1}_A)$)
\begin{align} \label{eq:BV1AdHA}
[\mathfrak{1}_A]_{BV}^\omega = \int_{\partial_*A} \omega(z) d\mathcal H^1(z).
\end{align}

\section{Proof of Lemma \ref{lem:keyLemma} and Lemma \ref{lem:keyLemma2}} \label{sec:keyLemmas}
Under the conditions of Theorem \ref{thm:dilationNumberVar} it was proved in Theorem 1.1 in \cite{HeWe} that the orthogonal polynomials enjoy an expansion
\begin{align} \label{eq:behavOPs}
    p_j(z) = n^{\frac14} \sqrt{\phi_\tau'(z)} \phi_\tau(z)^j e^{\frac12 n \mathcal Q_\tau(z)}
    \left(\sum_{k=0}^m n^{-k}\mathcal B_{\tau,k}(z)
    +\mathcal O(n^{-m-1})\right)
\end{align} 
as $j=\tau n\to\infty$, where the error term is uniform over all $z\in\mathbb C$ such that
\begin{align} \label{eq:dist}
    \operatorname{dist}(z, S_{Q,\tau}^c)=\mathcal O\left(\sqrt\frac{\log n}{n}\right)
\end{align}
as $j=\tau n\to\infty$, uniformly for $\tau$ in a fixed (small) interval around $\tau=1$. We have to explain several expressions in this formula. First, given $\tau$ (and under the condition that $\Delta Q>0$ on $S_Q$ and that $\tau$ is sufficiently close to $1$), $S_{Q,\tau}$ is the set
\begin{align*}
    S_{Q,\tau}=\{z\in\mathbb C : \check Q_\tau(z)=Q(z)\},
\end{align*}
where $\check Q_\tau$ solves the obstacle problem
\begin{align*}
    \check Q_\tau(z)=\sup\{Q(z) : q\in\operatorname{Subh}_\tau(\mathbb C)\text{ and }q\leq Q\}
\end{align*}
Here $\operatorname{Suhb}_\tau(\mathbb C)$ is the convex body of subharmonic functions in the plane which
grow at most like $\tau\log|z|^2$ as $z\to\infty$. Note in particular that $\check Q_1=\check Q$ and $S_{Q,1}=S_Q$. For $\tau$ near $1$, $S_{Q,\tau}$ should be thought of as a small perturbation of the droplet $S_Q$. The function $\phi_\tau$ now, is the conformal map from $S_{Q,\tau}^c$ to $\overline{\mathbb D}^c$, satisfying the (orthostatic) condition $\phi(\infty)=\infty$ and $\phi'(\infty)>0$. $\mathcal Q_\tau$ is the bounded holomorphic function on $S_{Q,\tau}^c$ such that $\Re \mathcal Q_\tau=Q$ on $\partial S_{Q,\tau}$ and $\Im \mathcal Q_\tau(\infty)=0$. The functions $\mathcal B_{\tau,k}$ are bounded and holomorphic on a neighborhood of $\overline{S_{Q,\tau}^c}$. For us, only $\mathcal B_{\tau,0}$ will be important. By Theorem 1.3 in \cite{HeWe} we have that
\begin{align*}
    \mathcal B_{\tau,0}(z) = (2\pi)^{-\frac14} e^{\mathcal H_{Q,\tau}(z)},
\end{align*}
where $\mathcal H_{Q,\tau}$ is the bounded holomorphic function on a neighborhood of $\overline{S_{Q,\tau}^c}$ satisfying $\Re \mathcal H_{Q,\tau}=\frac14 \log \Delta Q$ on $\partial S_{Q,\tau}$ and $\Im \mathcal H_{Q,\tau}(\infty)=0$. It was later shown by Hedenmalm \cite{Hedenmalm} using a soft Riemann-Hilbert problem that for $\tau=1$ the region of uniform convergence can be enlarged to $z\in\mathbb C$ such that
\begin{align*}
    \operatorname{dist}(z, S_{Q,\tau}^c)=\mathcal O(n^{-1/4}),
\end{align*}
which also holds for fixed $\tau=\frac{j}{n}$, but it is unclear whether \eqref{eq:behavOPs} would still hold with uniform error term for a range of $\tau$. Therefore, we will restrict to cases satisfying the condition in \eqref{eq:dist} in what follows. 
The result \eqref{eq:behavOPs} shows that the zeros of $p_n$ must be in the interior of the droplet, called the bulk.  
The large $n$ behavior of the orthogonal polynomials in the bulk is largely an open problem, and is only known in explicit cases. Close to the boundary, Hedenmalm and Wennman argued that effectively only the polynomial degrees $n-\sqrt n\log n\leq j\leq n-1$ contribute significantly to the kernel, with the error growing smaller than any inverse power of $n$. It thus suffices to consider the object
\begin{align*}
    \hat{\mathcal K_n}(z,w) = e^{-\frac12n \mathcal Q(z)} e^{-\frac12n \overline{\mathcal Q(w)}} \sum_{j=n-m_n}^{n-1} p_j(z) \overline{p_j(w)},
\end{align*}
where $m_n$ is defined as the integer part of $\sqrt n\log n$, and we define $\mathcal Q$ as the bounded holomorphic function such that $\Re \mathcal Q=Q$ on $\partial S_Q^c$ and $\Im \mathcal Q(\infty)=0$. Note that $\hat{\mathcal{K}}_n(z,w)$ is analytic in both $z$ and $\overline w$. Inserting \eqref{eq:behavOPs}, and reordering the summation indices, yields
\begin{align*}
    \hat{\mathcal K_n}(z,w) &= (1+\mathcal O(n^{-1})) (2\pi)^{-\frac12}\sqrt n\\
    &\times\sum_{j=1}^{m_n} e^{\frac12n(\mathcal Q_{(1-\tau)}(z)+2(1-\tau)\log \phi_{1-\tau}(z)-\mathcal Q(z))}
    e^{\frac12n(\overline{\mathcal Q_{(1-\tau)}(z)}+2(1-\tau)\log \overline{\phi_{1-\tau}(z)}-\overline{\mathcal Q(z)})},
\end{align*}
in a neighborhood of $\overline{S_Q^c\times S_Q^c}$. As explained in \cite{HeWe} an application of Hedenmalm and Shimorin \cite{HeSh} yields that the boundary of $S_{Q,\tau}$ depends real analytically on $1-\tau$ and furthermore (Theorem 6.4 in \cite{HeSh}) that $\phi_\tau$ depends real analytically on $1-\tau$. We will now explain why this is also true for $\mathcal Q_\tau$ and $\mathcal H_{Q,\tau}$. 

\begin{lemma}
    $\mathcal Q_\tau$ and $\mathcal H_{Q,\tau}$ depend real analytically on $\tau$ in a neighborhood of $\tau=1$. 
\end{lemma}

\begin{proof}
    By the Herglotz transform we have for $|z|>1$
    \begin{align*}
        \mathcal Q_\tau(\phi_\tau^{-1}(z)) = \int_{\partial\mathbb D} \frac{z+\zeta}{z-\zeta} Q(\phi_\tau^{-1}(\zeta)) d\mathcal H^1(z).
    \end{align*}
    Inserting an expansion in powers of $\tau-1$ for $\phi_\tau$ and using a Taylor series, we see that $\mathcal Q_\tau\circ\phi_\tau^{-1}$ is real analytic in $\tau$, on some neighborhood of $\tau=1$. Then composition with $\phi_\tau$ shows that the same is true for $\mathcal Q_\tau$. The argument is exactly the same for $\mathcal H_{Q,\tau}$. 
\end{proof}
In the sum describing our kernel, note that (after relabeling) we have $1-\tau=\frac{j}{n}$. We thus have, e.g., that under the condition $j\leq m_n$
\begin{align*}
    e^{\frac12n(\mathcal Q_{(1-\tau)}(z)+2(1-\tau)\log \phi_{1-\tau}(z)-\mathcal Q(z))} = e^{n A_0(z)+ j A_1(z) + \frac{j^2}{n} A_2(z)} (1+\mathcal O(\frac{\log^3 n}{\sqrt n}))
\end{align*}
as $n\to\infty$, for certain holomorphic functions $A_0, A_1$ and $A_2$ that exist in a neighborhood of $\overline{S_Q^c}$. Furthermore, in \cite[Section 5]{HeWe} (see equation (5.8)) it was determined exactly what the real part of these functions is in the case that $z$ is close to the droplet boundary. If $z_0\in \partial S_Q$ and $z=z_0+\frac{\vec n(z_0)\xi}{\sqrt{n\Delta Q(z_0)}}$ with $\xi=\mathcal O(\sqrt{\log n})$, then
\begin{align*}
    \left|e^{n A_0(z)+ j A_1(z) + \frac{j^2}{n} A_2(z)}\right|= \Delta Q(z_0)^\frac{1}{4} e^{-\left(\Re \xi+\frac{j|\phi'(z_0)|}{2\sqrt{n\Delta Q(z_0)}}\right)^2} (1+\mathcal O(\frac{\log^3n}{\sqrt n})),
\end{align*}
as $n\to\infty$.
In the case that $z=w$, they could then identify $\hat{\mathcal K}_n$ with a Riemann sum and obtain the complementary error function behavior (note that Hedenmalm and Wennman use a different definition of the error function). Note in particular that we have
\begin{align*}
    \Re A_0(z_0)=\Re A_1(z_0) = 0.
\end{align*}
Taking $\xi$ either purely real or purely imaginary shows that we also have $A_0'(z_0)=0$.
In our case, we need to allow for the case $z\neq w$, although we do impose that they are close to each other: $|z-w|=\mathcal O(\sqrt\frac{\log n}{n})$ as $n\to\infty$. We have
\begin{align*}
    e^{n A_0(z)+ j A_1(z) + \frac{j^2}{n} A_2(z)}
    e^{n \overline{A_0(w)}+ j \overline{A_1(w)} + \frac{j^2}{n} \overline{A_2(w)}}=\\
    e^{n(A_0(z)+\overline{A_0(w)})+ j(A_1(z)+\overline{A_1(w)}) + \frac{j^2}{n} (A_2(z)+\overline{A_2(w)})}
\end{align*}
and the condition on $z-w$ means that $A_0(z)+\overline{A_0(w)}, A_1(z)+\overline{A_1(w)}$ and $A_2(z)+\overline{A_2(w)}$ are close to being real. The first term is not so important since it drops out of the entire sum, but for the other two terms we have for $z-w=\mathcal O(\sqrt\frac{\log n}{n})$
\begin{align*}
    j(A_1(z)+\overline{A_1(w)}) &= 
    2j \Re A_1(z)+j \overline{A_1'(z)(w-z)}+\mathcal O(\frac{\log^3 n}{\sqrt n}),\\
    \frac{j^2}{n} (A_2(z)+\overline{A_2(w)}) &= 2\frac{j^2}{n} \Re A_2(z)+\mathcal O(\frac{\log^3 n}{\sqrt n})
\end{align*}
as $n\to\infty$ and $j\leq \sqrt n \log n$ (with even smaller error terms actually). We can do the same argument with $w$ as base point instead. Plugging in 
\begin{align*}
    z=z_0+\frac{\vec n(z_0)\xi}{\sqrt{n\Delta Q(z_0)}}, \quad
    w=w_0+\frac{\vec n(w_0)\eta}{\sqrt{n\Delta Q(w_0)}},
\end{align*}
where $\xi,\eta=\mathcal O(\sqrt{\log n})$, we find that 
\begin{align*}
    \hat{\mathcal K_n}(z,w) =& (1+\mathcal O(\frac{\log^3 n}{\sqrt n})) \sqrt n(2\pi)^{-\frac12} \Delta Q(z_0)^\frac{1}{4} \Delta Q(w_0)^\frac{1}{4} e^{n(A_0(z)+\overline{A_0(w)})}\\
    &\times\sum_{j=1}^{\sqrt n\log n} \sqrt{\phi'(z_0)} \overline{\sqrt{\phi'(w_0)}} 
     e^{-\left(\Re \xi+\frac{j|\phi'(z_0)|}{2\sqrt{n\Delta Q(z_0)}}\right)^2
     -\left(\Re \eta+\frac{j|\phi'(w_0)|}{2\sqrt{n\Delta Q(w_0)}}\right)^2+i\sqrt n \alpha \frac{j|\phi'(z_0)|}{2\sqrt{n\Delta Q(z_0)}}}\\
     =& (1+\mathcal O(\frac{\log^3 n}{\sqrt n})) \sqrt 2 n\pi^{-\frac12} \Delta Q(z_0)^\frac{1}{2} \Delta Q(w_0)^\frac{1}{2}e^{n(A_0(z)+\overline{A_0(w)})}\\
    &\times\sum_{j=1}^{\sqrt n\log n} \frac{|\phi'(z_0)|}{2\sqrt{n\Delta Q(z_0)}}
     e^{-\left(\Re \xi+\frac{j|\phi'(z_0)|}{2\sqrt{n\Delta Q(z_0)}}\right)^2
     -\left(\Re \eta+\frac{j|\phi'(z_0)|}{2\sqrt{n\Delta Q(z_0)}}\right)^2+i\sqrt n \alpha \frac{j|\phi'(z_0)|}{2\sqrt{n\Delta Q(z_0)}}},
\end{align*}
where
\begin{align*}
    \alpha = 2\frac{\sqrt{\Delta Q(z_0)}}{|\phi'(z_0)|}\Im\left(A_1'(z)(z-w))\right).
\end{align*}
\begin{lemma} \label{lem:RiemannSum}
    We have uniformly for $\xi,\eta\in\mathbb C$ and $\alpha\in\mathbb R$ that
    \begin{multline*}
        \sum_{j=1}^{m_n} \frac{|\phi'(z_0)|}{2\sqrt{n\Delta Q(z_0)}}
     e^{-\left(\Re \xi+\frac{j|\phi'(z_0)|}{2\sqrt{n\Delta Q(z_0)}}\right)^2
     -\left(\Re \eta+\frac{j|\phi'(z_0)|}{2\sqrt{n\Delta Q(z_0)}}\right)^2+i\sqrt n \alpha \frac{j|\phi'(z_0)|}{2\sqrt{n\Delta Q(z_0)}}}\\
     = \int_0^{\infty} e^{-(\Re \xi+t)^2-(\Re \eta+t)^2+i\sqrt n \alpha t} dt
     -\frac{|\phi'(z_0)|}{4\sqrt{n Q(z_0)}} e^{-(\Re \xi)^2-(\Re \eta)^2}
     \\
     + \mathcal O(\frac{(|\Re\xi|+|\Re\eta|+\sqrt n|\alpha|+\log n)^2}{n}\log n).
    \end{multline*}
\end{lemma}

\begin{proof}
    Let us define
    \begin{align*}
        f_n(t)=e^{-(\Re\xi+\gamma_0 t)^2-(\Re\eta+\gamma_0 t)^2+i\sqrt n\alpha \gamma_0 t},
        \quad \gamma_0 = \frac{|\phi'(z_0)|}{2\sqrt{\Delta Q(z_0)}},
    \end{align*}
    and $g_n(t)=f_n(t \log n)$. By bounding the remainder in the second order truncation of the Euler-Maclaurin formula (e.g., see \cite[Chapter 5]{Katkinson}), we have the general estimate
    \begin{align*}
        \left|\frac1{m_n} \sum_{j=1}^{m_n} g_n\left(\frac{j}{m_n}\right) + \frac{g_n(1)-g_n(0)}{2m_n} - \int_0^1 g_n(t) dt\right|
        \leq \frac{1}{12 m_n^2} \sup_{[0,1]} |g_n''|.
    \end{align*}
Then it follows that
\begin{align*}
    &\left|\frac1{\sqrt n} \sum_{j=1}^{m_n} f_n\left(\frac{j}{\sqrt n}\right)+ \frac{g_n(1)-g_n(0)}{2\sqrt n}-\int_0^{\log n} f_n(t) dt\right|
    \leq \frac{\log n}{12 m_n^2} \sup_{[0,1]} |g_n''|\\
    &= \frac{\log^3 n}{12 m_n^2} \sup_{[0,\log n]} |f_n''|
    =  \mathcal O(\frac{(|\Re\xi|+|\Re\eta|+\sqrt n|\alpha|+\log n)^2}{n}\log n).
\end{align*}
It follows from the behavior of the complementary error function that changing the integration bound from $\log n$ to $\infty$ yields a contribution that goes to $0$ faster than any power of $n$. Finally, a substitution $t\to \gamma_0^{-1}t$ yields the result.
\end{proof}
This limiting formula can be rewritten, and we have that
\begin{multline*}
    \frac{2\sqrt 2}{\sqrt \pi}\int_0^{\infty} e^{-(\Re \xi+t)^2-(\Re \eta+t)^2+i\sqrt n \alpha t} dt = \\
    \frac{2\sqrt 2}{\sqrt \pi} 
    e^{-\frac12(\Re \xi)^2-\frac12(\Re \eta)^2-\frac12(\Re\xi+\Re\eta)i\sqrt n \alpha+n\frac{\alpha^2}{4}}
    \int_0^\infty e^{-2(t+\frac12\Re\xi+\frac12\Re\eta-\frac14 i\sqrt n \alpha)^2} dt\\
    = 
    e^{-\frac12(\Re \xi)^2-\frac12(\Re \eta)^2-\frac12(\Re\xi+\Re\eta)i\sqrt n \alpha+n\frac{\alpha^2}{8}}
    \operatorname{erfc}\left(\frac{\Re\xi+\Re\eta}{\sqrt 2}-i\sqrt n\frac{\alpha}{2\sqrt 2}\right).
\end{multline*}
Comparing with the known result in \cite{HeWe} for $z_0=w_0$ and $\alpha=0$ we have
\begin{align*}
    n\Re A_0(z)=\Re\left(A_0''(z_0) \frac{\vec n(z_0)^2\xi^2}{2n\Delta Q(z_0)}\right)+\mathcal O(\frac{\log^3 n}{\sqrt n}).
\end{align*}
Using the expression in \cite{HeWe} in the case $\alpha=0$ and $z_0=w_0$ we infer that
\begin{align*}
    \left|e^{n(A_0(z)+\overline{A_0(w)})}\right|
    = e^{-\frac12|\xi-\eta|^2+\frac12(\Re\xi)^2+\frac12(\Re\eta)^2} (1+\mathcal O(\frac{\log^3 n}{\sqrt n})),
\end{align*}
as $n\to\infty$, for $\xi,\eta=\mathcal O(\sqrt{\log n})$ and $|z_0-w_0|=\mathcal O(\sqrt\frac{\log n}{n})$.
Thus, taking the modulus of $\hat{\mathcal K_n}(z,w)$, we obtain 
\begin{align*}
    \frac1{n\sqrt{\Delta Q(z_0)\Delta Q(w_0)}}\left|\hat{\mathcal K_n}\left(z_0+\frac{\vec n(z_0) \xi}{\sqrt{n \Delta Q(z_0)}},w_0+\frac{\vec n(w_0)\eta}{\sqrt{n \Delta Q(w_0)}}\right)\right| = \\
        (1+\mathcal O(\frac{\log^3 n}{\sqrt n}))
        e^{-\frac12|\xi-\eta|^2+n\frac{\alpha^2}{8}}
        \left|\operatorname{erfc}\left(\frac{\Re\xi+\Re\eta}{\sqrt 2}-i\sqrt n\frac{\alpha}{2\sqrt 2}\right)\right|.
\end{align*}
Since the difference between $|\hat{\mathcal K}_n|$ and $|\mathcal K_n|$ is smaller than any negative power of $n$, the same formula holds for $\mathcal K_n$, that is, uniformly for $|z_0-w_0|=\mathcal O(\sqrt\frac{\log n}{n})$ and $\xi,\eta=\mathcal O(\sqrt n)$ we have
\begin{align} \nonumber
    \frac1{n\sqrt{\Delta Q(z_0)\Delta Q(w_0)}}\left|\mathcal K_n\left(z_0+\frac{\vec n(z_0) \xi}{\sqrt{n \Delta Q(z_0)}},w_0+\frac{\vec n(w_0)\eta}{\sqrt{n \Delta Q(w_0)}}\right)\right| = \\ \label{eq:eqLimitAlpha}
        (1+\mathcal O(\frac{\log^3 n}{\sqrt n}))
        e^{-\frac12|\xi-\eta|^2+n\frac{\alpha^2}{8}}
        \left|\operatorname{erfc}\left(\frac{\Re\xi+\Re\eta}{\sqrt 2}-i\sqrt n\frac{\alpha}{2\sqrt 2}\right)\right|
\end{align}
as $n\to\infty$. The remaining question is what the value of $\alpha$ is.  

\begin{proof}[Proof of Lemma \ref{lem:keyLemma}]
    Let us consider the case that $z_0-w_0=\mathcal O(\sqrt{\frac{\log n}{n}})$ and $\xi,\eta=\mathcal O(\sqrt{\log n})$. For $z_0=w_0$ we know from the locally uniform limit in \cite{HeWe} that $-\frac12\sqrt n\alpha=\Im\xi\textcolor{red}{-}\Im\eta$ and thus Lemma \ref{lem:keyLemma} holds, since we just proved \ref{eq:eqLimitAlpha}. Even when $z_0\neq w_0$, we can relate this to the case $z_0=w_0$. Namely, we may write that 
    \begin{align*}
        w_0+\frac{\vec n(w_0) \eta}{\sqrt{n\Delta Q(w_0)}}
        = z_0+\frac{\vec n(z_0) \tilde\eta}{\sqrt{n\Delta Q(z_0)}},
    \end{align*}
    where
    \begin{align*}
        \tilde\eta &= \sqrt\frac{\Delta Q(z_0)}{\Delta Q(w_0)} \frac{\vec n(w_0)}{\vec n(z_0)}\eta+\sqrt {n\Delta Q(z_0)}\frac{w_0-z_0}{\vec n(z_0)}\\
        &= \eta+\sqrt {n\Delta Q(z_0)}\frac{w_0-z_0}{\vec n(z_0)}+\mathcal O(\frac{\log n}{\sqrt n}).
    \end{align*}
    Note that we indeed have $\tilde\eta=\mathcal O(\sqrt{\log n})$ as $n\to\infty$.
    We see that
    \begin{align*}
        \xi+\overline{\tilde\eta}
        = \xi+\overline\eta+\vec n(z_0) \sqrt{n\Delta Q(z_0)} (\overline{w_0-z_0})
        + \mathcal O(\frac{\log n}{\sqrt n}).
    \end{align*}
    Explicitly, the outward unit normal vector is given by
    \begin{align*}
        \vec n(z_0) = \frac{\phi(z_0)/\phi'(z_0)}{|\phi(z_0)/\phi'(z_0)|}
        = \phi(z_0)\frac{\overline{\phi'(z_0)}}{|\phi'(z_0)|}.
    \end{align*}
    On the other hand, We have that
    \begin{align*}
        w_0-z_0 &= \phi^{-1}(\phi(w_0))-\phi^{-1}(\phi(z_0))
        =(\phi^{-1})'(\phi(z_0)) (\phi(w_0)-\phi(z_0))
    +\mathcal O(\sqrt\frac{\log n}{n})  \\
        &= \frac{\phi(z_0)}{\phi'(z_0)}\left(e^{\log(\overline{\phi(z_0)}\phi(w_0))}-1\right) +\mathcal O(\sqrt\frac{\log n}{n})\\
        &= \frac{\phi(z_0)}{\phi'(z_0)} \log(\overline{\phi(z_0)}\phi(w_0)) +\mathcal O(\sqrt\frac{\log n}{n}).
    \end{align*}
    Thus
    \begin{align*}
        \vec n(z_0)  (\overline{w_0-z_0})=\frac{1}{|\phi'(z_0)|} \log(\phi(z_0)\overline{\phi(w_0)})+\mathcal O(\sqrt\frac{\log n}{n}).
    \end{align*}
    Comparing this with \eqref{eq:eqLimitAlpha}, we infer that
    \begin{align*}
        -i \sqrt n\frac{\alpha}{2} = i\Im\xi-i\Im\eta+ \sqrt{n\Delta Q(z_0)}\frac{\log(\phi(z_0)\overline{\phi(w_0)})}{|\phi'(z_0)|}+\mathcal O(\frac{\log n}{\sqrt n})
    \end{align*}
    as $n\to\infty$. A first order Taylor approximation shows us that we may effectively drop the $\mathcal O$ term.
\end{proof}

The arguments we used to prove Lemma \ref{lem:keyLemma} essentially also provide us with all the ingredients to prove Lemma \ref{lem:keyLemma2}. Analogous to \cite{AmCr} we start with a summation by parts and rewrite
\begin{multline*}
    \hat{\mathcal K}_n(z,w) = \frac1{\phi(z)\overline{\phi(w)}-1} \times\\ \left(a_{n-1}(\phi(z)\overline{\phi(w)})^n+\sum_{j=n-m_n}^{n-2}(a_{j+1}-a_j) (\phi(z)\overline{\phi(w)})^{j+1}-a_m(\phi(z)\overline{\phi(w)})^{m+1}\right)
\end{multline*}
where 
\begin{align*}
    a_j=a_j(z,w) = \left(\frac{\phi_{1-j/n}(z)}{\phi(z)}\right)^j e^{n(\mathcal Q_{1-j/n}(z)-Q(z))}
    \overline{\left(\frac{\phi_{1-j/n}(w)}{\phi(w)}\right)^j e^{n(\mathcal Q_{1-j/n}(w)-Q(w))}}.
\end{align*}
A couple of things were proved about the $a_j$ in \cite{AmCr} (Section 4.3). They imply the following estimates:
\begin{align*}
    a_{n-1} = 1+\mathcal O(n^{-1})\quad \text{ and }\quad a_{n-m_n} = e^{-c\log^2n}
\end{align*}
as $n\to\infty$, where $c>0$ is uniform. Furthermore, it follows from Lemma 4.4 in \cite{AmCr} that 
\begin{align*}
    a_{j+1}-a_j = a_{j+1} \mathcal O(\frac{n-j}{n}+\frac{(n-j)^2}{n^2})
    = a_{j+1} \mathcal O(\frac{n-j}{n}+\frac{\log^2n}{n})
\end{align*}
as $n\to\infty$, as long as $\frac{j}{n}$ is bounded from below by a positive constant.

\begin{proof}[Proof of Lemma \ref{lem:keyLemma2}] Let us assume that $\xi,\eta=\mathcal O(\sqrt{\log n})$ as $n\to\infty$, and $z_0,w_0\in\partial S_Q$. 
    We notice that as $n\to\infty$
    \begin{multline*}
        \sum_{j=n-m_n}^{n-2}(a_{j+1}-a_j) (\phi(z)\overline{\phi(w)})^{j+1}
        = \sum_{j=1}^{m_n-1} \mathcal O(\frac{j}{n})a_{n-j} (\phi(z)\overline{\phi(w)})^{n-j}+\mathcal O(\frac{\log^3n}{\sqrt n})\\
        =\frac1{\sqrt n}\sum_{j=1}^{m_n-1} \mathcal O(\frac{j}{\sqrt n}) e^{n(A_0(z)+\overline{A_0(w)})+ j(A_1(z)+\overline{A_1(w)}) + \frac{j^2}{n} (A_2(z)+\overline{A_2(w)})}+\mathcal O(\frac{\log^3n}{\sqrt n}).
    \end{multline*}
    Using Cauchy-Schwarz we find 
    \begin{multline*}
        \left|\sum_{j=n-m_n}^{n-2}(a_{j+1}-a_j) (\phi(z)\overline{\phi(w)})^{j+1}\right|^2\\
        \lesssim
        \left(\frac1{\sqrt n}\sum_{j=1}^{m_n} \frac{j}{\sqrt n} e^{2n \Re A_0(z)+2j\Re A_1(z)+2\frac{j^2}{n}\Re A_2(z)}\right)
        \left(\frac1{\sqrt n}\sum_{j=1}^{m_n} \frac{j}{\sqrt n} e^{2n \Re A_0(w)+2j\Re A_1(w)+2\frac{j^2}{n}\Re A_2(w)}\right)\\
        \lesssim \left(\frac{|\phi'(z_0)|}{2\sqrt{n\Delta Q(z_0)}}\sum_{j=1}^{m_n} \frac{j|\phi'(z_0)|}{2\sqrt{n\Delta Q(z_0)}}
     e^{-2\left(\Re \xi+\frac{j|\phi'(z_0)|}{2\sqrt{n\Delta Q(z_0)}}\right)^2}\right)\\
     \left(\frac{|\phi'(w_0)|}{2\sqrt{n\Delta Q(w_0)}}\sum_{j=1}^{\sqrt n\log n} \frac{j|\phi'(w_0)|}{2\sqrt{n\Delta Q(w_0)}}
     e^{-2\left(\Re \eta+\frac{j|\phi'(w_0)|}{2\sqrt{n\Delta Q(w_0)}}\right)^2}\right)\\
     \lesssim \int_0^\infty t e^{-2(\Re\xi+t)^2} dt
     \int_0^\infty t e^{-2(\Re\eta+t)^2} dt,
    \end{multline*}
    where the last step follows by arguments analogous to those in the proof of Lemma \ref{lem:RiemannSum}. 
    Indeed, we have that
    \begin{align*}
        \int_0^\infty t e^{-2(\Re\xi+t)^2} dt =
        \frac14 e^{-2(\Re \xi)^2}-\frac{\sqrt\pi}{2\sqrt 2} \Re \xi\operatorname{erfc}(\sqrt 2\Re\xi)
        = \mathcal O(e^{-2(\Re\xi)^2}).
    \end{align*}
    The reader may verify that a similar behavior holds for the remaining $(\phi(z)\overline{\phi(w)})^n$ term, and the lemma follows.
\end{proof}

\section{Proof of the bulk theorem: a norm representation result} \label{sec:proofMain}

The aim of this section is to prove a norm representation result for functions of weighted bounded variation. For smooth functions in $W^{1,1}$ a seminal result on norm representations of functions in Sobolev spaces is due to Bourgain, Brezis and Mironescu \cite{BoBrMi}. The case of unweighted bounded variation functions and radial kernels was studied by Dávila \cite{Davila}. We are dealing with functions of weighted bounded variation, and our kernel is, in general, not (asymptotically) radial. To this end, we prove the following result, from which Theorem \ref{thm:main} directly follows. 

\begin{proposition}\label{teo:BV}
Let $f:\mathbb C\to\mathbb R$ be a compactly supported function in $L^1(\mathbb C).$ Then 
\begin{equation*}
\lim_{n\to \infty} \frac1{\sqrt n} \int_{\mathbb C} \int_{\mathbb C} |f(z)-f(w)| |\mathcal K_n(z,w)|^2 dA(z) dA(w)
= \frac{1}{2 \pi \sqrt{\pi}} [f]_{BV}^{\sqrt{\Delta Q}}.
\end{equation*}
\end{proposition}

Throughout this section, we fix a compact set $K\subset \mathring S_Q$, and we let $\Omega\subset\mathring S_Q$ be an open neighborhood of $K$. To avoid dependence on $\Omega$ rather than $K$ in what follows, let us simply agree that $\Omega$ is the open set consisting of all points in $S_Q$ with distance bigger than $\frac12 d_K$ to $\partial S_Q$, i.e., half the distance between $K$ and $\partial S_Q$. We may and shall assume that $\Delta Q>0$ on $\overline \Omega$. In what follows, let
\begin{align}
\delta_n = \delta\frac{\log n}{\sqrt n}
\end{align}
for some $\delta>0$. We shall pick $\delta>\frac3{2\epsilon}$, where $\epsilon$ is the constant depending on $Q$ and $K$ in the following lemma (which is valid for any $\delta>0$). In this section, the $L^1$ norm is defined in the usual way
\begin{align*}
    \lVert f\rVert_{L^1} = \int_{\mathbb C} |f(x+i y)| dx dy.
\end{align*}
(In particular, there is no extra factor $1/\pi$ as in the definition of $dA(z)$.)

\begin{lemma} \label{lem:epsilonDelta}
Let $f\in L^1(\mathbb C)$ have support in $K$. Then there exists a constant $\epsilon>0$ depending only on $Q$ and $K$ such that uniformly for $n=1,2,\ldots$
\begin{multline*}
\int_{\mathbb C^2} |f(z)-f(w)| |\mathcal K_n(z,w)|^2 dA(z) dA(w)\\ 
= \int_{\Omega} \int_{B(w,\delta_n)} |f(z)-f(w)| |\mathcal K_n(z,w)|^2 dA(z) dA(w) 
+ \lVert f\rVert_{L^1} \mathcal O(n^{2-\epsilon \delta})
\end{multline*}
where the constant implied by the $\mathcal O$ term depends only on $K$ and $Q$. 
\end{lemma}

\begin{proof}
For $n$ big enough, since $f$ has support in $K$, we have trivially
\begin{align*}
\int_{\Omega^c} \int_{B(w,\delta_n)} |f(z)-f(w)| |\mathcal K_n(z,w)|^2 dA(z) dA(w) = 0.
\end{align*}
Now notice using the symmetry of $|\mathcal K_n|$ that
\begin{multline*}
\int_{\mathbb C} \int_{B(w,\delta_n)^c} |f(z)-f(w)| |\mathcal K_n(z,w)|^2 dA(z) dA(w) 
\\
\leq 2 \int_{\mathbb C} |f(w)| \int_{B(w,\delta_n)^c} |\mathcal K_n(z,w)|^2 dA(z) dA(w)\\
= 2 \int_K |f(w)| \mathcal K_n(w,w) \int_{B(w,\delta_n)^c} B_n^{(w)}(z) dA(z) dA(w),
\end{multline*}
where we used that the support of $f$ is contained in $K$. 
Now we insert the estimate \eqref{eq:Knzw} and obtain uniformly for $w\in K$
\begin{align*}
\int_{B(w,\delta_n)^c} B_n^{(w)}(z) dA(z) 
&\leq C n \int_{B(w,\delta_n)^c} e^{-\epsilon \sqrt n \delta_n} e^{-\frac12n(Q(z)-\check Q(z))} dA(z)\\
&= C n^{1-\epsilon \delta} \int_{B(0,\delta_n)^c} e^{-\frac12n(Q(w+z)-\check Q(w+z))} dA(z)
\end{align*}
For $R>0$ big enough we have $Q(z+w)-\check Q(z+w) \geq C_Q \log |z|$ for all $|z|\geq R$ and all $w\in K$, where $C_Q>0$ can be picked independent of $w$, while on the region $|z|\leq R$ we may simply use that $Q\geq \check Q$. Inserting these behaviors in the integral and using that $\mathcal K_n(w,w)\leq C n$ for some constant $C>0$ for all $w\in K$ (equation \eqref{eq:KnzzC}), we arrive at the result. 
\end{proof}

\begin{lemma}\label{lem:ineq_middle}
Let $f:\mathbb C\to\mathbb R$ be a smooth function with support in $K$. Then we have
\begin{multline*}
\int_\Omega \int_{B(w,\delta_n)} |f(z)-f(w)| |\mathcal K_n(z,w)|^2 dA(z) dA(w)
= (1+\mathcal O(1/\sqrt n))\\
\int_\Omega \mathcal K_n(w,w) \int_{B(0, \delta_n \sqrt{n\Delta Q(w)})} \left|f\left(w+\frac{z}{\sqrt{n \Delta Q(w)}}\right)-f(w)\right| e^{-|z|^2} dA(z) dA(w)\end{multline*}
as $n\to\infty$, where the implied constant is independent of $f$. 
\end{lemma}

\begin{proof}
We note that we can rewrite our integral as
\begin{multline*}
\int_\Omega \int_{B(w,\delta_n)} |f(z)-f(w)| |\mathcal K_n(z,w)|^2 dA(z) dA(w)\\
= \int_\Omega \mathcal K_n(w,w) \int_{B(w, \delta_n)} |f(z)-f(w)|  B_n^{(w)}(z) dA(z) dA(w).
\end{multline*}
Plugging the behavior \eqref{eq:behavBerezin} into our integral, we get, after translating and rescaling
\begin{multline*}
\int_\Omega \int_{B(w,\delta_n)} |f(z)-f(w)| |\mathcal K_n(z,w)|^2 dA(z) dA(w)\\
= 
\int_\Omega \mathcal K_n(w,w) \int_{B(0, \delta_n \sqrt{n\Delta Q(w)})} \left|f\left(w+\frac{z}{\sqrt{n \Delta Q(w)}}\right)-f(w)\right|\\
(e^{-|z|^2}(1+\mathcal O(|z|^3/\sqrt n)+\mathcal O(1/n))+\mathcal O(n^{-3})) dA(z) dA(w).
\end{multline*}
To investigate the term with $\mathcal O(n^{-3})$ behavior, we will use that there exists a constant $C>0$ (independent of $f$) such that
\begin{align*}
\mathcal K_n(w,w)&\leq C n, \qquad w\in\mathbb C,\\
\int_{\mathbb C} |f(w+z)-f(w)| dA(w) &\leq C \lVert \nabla f\rVert_{L^1} |z|, \qquad z\in\mathbb C.
\end{align*}
(see \eqref{eq:KnzzC} for the first estimate). The second estimate was proved in the previous lemma.
Using these inequalities we get
\begin{align*}
\int_\Omega \mathcal K_n(w,w) \int_{B(0, \delta_n \sqrt{n\Delta Q(w)})} \left|f\left(w+\frac{z}{\sqrt{n \Delta Q(w)}}\right)-f(w)\right| dA(z) dA(w)\\
\leq C n (\max_{\overline\Omega} \Delta Q)\int_\Omega \int_{B(0, \delta \log n)} \left|f\left(w+\frac{z}{\sqrt{n}}\right)-f(w)\right| dA(z) dA(w)\\
\leq C^2 \sqrt n (\max_{\overline\Omega} \Delta Q) \lVert \nabla f\rVert_{L^1} \int_{B(0,\delta\log n)} |z| dA(z).
\end{align*}
That we can absorb the error term into the $\mathcal O(1/\sqrt n)$ term will follow from Lemma \ref{Lemma:limit} below.
\end{proof}

\begin{lemma}\label{Lemma:limit}
Let $f:\mathbb C\to\mathbb R$ be a smooth function with support in $K$. Then we have
\begin{align*}
\lim_{n\to\infty} \frac1{\sqrt n}\int_\Omega \mathcal K_n(w,w) \int_{B(0, \delta_n \sqrt{n\Delta Q(w)})} \left|f\left(w+\frac{z}{\sqrt{n \Delta Q(w)}}\right)-f(w)\right| e^{-|z|^2} dA(z) dA(w)\\
= \frac{1}{2\pi\sqrt\pi} \lVert \nabla f \sqrt{\Delta Q}\rVert_{L^1}.
\end{align*}
\end{lemma}

\begin{proof}
Let us define 
\begin{align*}
g_n(w,z) = \frac1{\sqrt n}\mathcal K_n(w,w) \left|f\left(w+\frac{z}{\sqrt{n \Delta Q(w)}}\right)-f(w)\right| e^{-|z|^2} \mathfrak{1}_{B(0,\delta_n \sqrt{n\Delta Q(w)})}(z).
\end{align*}
It converges pointwise, namely, for any $(w,z)\in \Omega\times\mathbb C$ 
\begin{align*}
\lim_{n\to\infty} g_n(w,z) = \sqrt{\Delta Q(w)} | \nabla f(w)\cdot (\cos \theta, \sin\theta)| |z| e^{-|z|^2},
\end{align*}
where $z=|z| e^{i\theta}$. 
Furthermore, by the arguments in the proof of the preceding lemmas we have
\begin{align*}
g_n(w,z) \leq C \sqrt{\Delta Q(w)} \int_0^{\sqrt{c_Q}} \lVert \nabla f(w+z t/\sqrt n)\rVert dt |z| e^{-|z|^2} dA(z),
\end{align*}
which has a finite integral over $\Omega\times\mathbb C$. Namely, using that $\sqrt{\Delta Q}$ is bounded on $\Omega$, it suffices to note that
\begin{multline*}
    \int_{\Omega} \int_{\mathbb C} \int_0^{\sqrt{c_Q}} \lVert \nabla f(w+z t/\sqrt n)\rVert dt |z| e^{-|z|^2} dA(z) dA(w)\\
    \leq \int_0^{\sqrt{c_Q}} \int_{\mathbb C} \int_{\mathbb C}  \lVert \nabla f(w+z t/\sqrt n)\rVert |z| e^{-|z|^2} dA(w) dA(z) dt
    = \frac12 \sqrt\pi \sqrt{c_Q} \lVert \nabla f\rVert_{L^1}.
\end{multline*}
Hence the dominated convergence theorem implies that
\begin{multline*}
\lim_{n\to\infty}\frac1{\sqrt n}\int_\Omega \mathcal K_n(w,w) \int_{B(0, \delta_n \sqrt{n\Delta Q(w)})} \left|f\left(w+\frac{z}{\sqrt{n \Delta Q(w)}}\right)-f(w)\right| e^{-|z|^2} dA(z) dA(w)\\
= \int_\Omega \int_{\mathbb C} \sqrt{\Delta Q(w)} | \nabla f(w)\cdot (\cos\theta, \sin\theta)| |z| e^{-|z|^2} dA(z) dA(w)\\
= \int_\Omega \int_{0}^\infty \frac{4}{\pi} \sqrt{\Delta Q(w)} \lVert\nabla f(w)\rVert r^2 e^{-r^2} dr dA(w)
= \frac{1}{2\pi\sqrt\pi} \lVert \nabla f \sqrt{\Delta Q}\rVert_{L^1}.
\end{multline*}
\end{proof}

\begin{corollary} \label{cor:mainThmSmooth}
Let $f:\mathbb C\to\mathbb R$ be a smooth function with support in $K$. Then we have
\begin{align*}
\lim_{n\to\infty} \frac1{\sqrt n} \int_{\mathbb C^2} |f(z)-f(w)||\mathcal K_n(z,w)|^2 dA(z) dA(w)
= \frac{1}{2\pi\sqrt\pi} \lVert \nabla f \sqrt{\Delta Q}\rVert_{L^1}.
\end{align*}
\end{corollary}

\begin{lemma} \label{eq:fSmoothUpperBound}
Let $f:\mathbb C\to\mathbb R$ be a smooth function with support in $K$. Then we have
\begin{align*}
\frac1{\sqrt n} \int_\Omega \mathcal K_n(w,w) \int_{B(0, \delta_n \sqrt{n\Delta Q(w)})} \left|f\left(w+\frac{z}{\sqrt{n \Delta Q(w)}}\right)-f(w)\right| e^{-|z|^2} dA(z) dA(w)\\
\leq (\frac{1}{2\pi \sqrt\pi}+C/\sqrt n) \lVert \nabla f \sqrt{\Delta Q}\rVert_{L^1},
\end{align*}
where $C>0$ is a constant independent of $f$. 
\end{lemma}

\begin{proof}
First we write
\begin{multline*}
\frac1{\sqrt n}\int_\Omega \mathcal K_n(w,w) \int_{B(0, \delta_n \sqrt{n\Delta Q(w)})} \left|f\left(w+\frac{z}{\sqrt{n \Delta Q(w)}}\right)-f(w)\right| e^{-|z|^2} dA(z) dA(w)\\
= \frac1\pi\int_\Omega \frac1n \mathcal K_n(w,w) \int_0^{\delta \log n} \int_0^{2\pi} 
\left|\int_0^1 \nabla f(w+\frac1{\sqrt n} r e^{i\theta} t) \cdot (\cos\theta,\sin\theta) dt\right| \\
 e^{-\Delta Q(w) r^2} r^2 dr d\theta \Delta Q(w) dA(w)\\
 \leq \frac1\pi\int_\Omega \frac1n \mathcal K_n(w,w) \int_0^{\delta \log n} \int_0^{2\pi} 
\int_0^1 |\nabla f(w+\frac1{\sqrt n} r e^{i\theta} t) \cdot (\cos\theta,\sin\theta)| dt \\
 e^{-\Delta Q(w) r^2} r^2 dr d\theta \Delta Q(w) dA(w)
\end{multline*}
Using Fubini, we may interchange the order of the integrations over $dt d|z| d\theta$ and $dA(w)$. Then a translation yields that the remaining integral equals
\begin{multline*}
\frac1\pi\int_{0}^{2\pi} \int_0^{\delta \log n} \int_0^1 \int_{\Omega+\frac1{\sqrt n} z t} |\nabla f(w) \cdot (\cos\theta,\sin\theta)| \frac1n\mathcal K_n(w-\frac1{\sqrt n} z t,w-\frac1{\sqrt n} z t)\\
\Delta Q(w-\frac1{\sqrt n} z t)
e^{-\Delta Q(w-\frac1{\sqrt n} z t) |z|^2} dA(w) dt |z|^2 d|z| d\theta
\end{multline*}
Let $K'\subset \mathring S_Q$ be any compact set that contains $\Omega$. For all $w\in K'$, we have uniformly as $n\to\infty$ that
\begin{align*}
\Delta Q(w-\frac1{\sqrt n}) &= \Delta Q(w)+\mathcal O(|z|/\sqrt n),\\
\frac1n \mathcal K_n(w-\frac1{\sqrt n} z t,w-\frac1{\sqrt n} z t)
&= \Delta Q(w) + \mathcal O(|z|/\sqrt n),\\
e^{-\Delta Q(w-\frac1{\sqrt n} z t) |z|^2} &= e^{-\Delta Q(w) |z|^2} (1+\mathcal O(|z|^3/\sqrt n)),
\end{align*} 
where the implied constants depend only on $Q$ and $K'$. Plugging these behaviors in, and then replacing $\Omega+\frac1{\sqrt n} z t$ by $K'$, and using Fubini, our integral is $\leq$ the integral
\begin{multline*}
\frac1\pi\int_{K'} \int_0^{\delta \log n} \int_{0}^{2\pi} |\nabla f(w) \cdot (\cos\theta,\sin\theta)| d\theta\\
(\Delta Q(w)^2+\mathcal O(\frac{|z|+|z|^3}{\sqrt n})) e^{-\Delta Q(w)|z|^2} |z|^2 d|z|  dA(w),
\end{multline*}
where we already carried out the integration over $dt$. Now using
\begin{align*}
\int_0^{2\pi} |\nabla f(w)\cdot (\cos\theta, \sin\theta)| d\theta
= 4 \lVert \nabla f(w)\rVert,
\end{align*}
our remaining integral is
\begin{align*}
\frac2\pi \int_{K'} & \int_0^{\delta \log n} \lVert \nabla f(w)\rVert
(\Delta Q(w)^2+\mathcal O(\frac{|z|+|z|^3}{\sqrt n})) e^{-\Delta Q(w)|z|^2} |z|^2 d|z|  dA(w)\\
&\leq \frac1{2\sqrt\pi}\int_{K'}  \lVert \nabla f(w)\rVert (\sqrt{\Delta Q(w)}+\mathcal O(1/\sqrt n)) dA(w)\\
&= \frac1{2\pi\sqrt\pi} \lVert \nabla f \sqrt{\Delta Q}\rVert_{L^1}
+ \lVert \nabla f\rVert_{L^1}\mathcal O(1/\sqrt n)
\end{align*}
as $n\to\infty$, from which the lemma follows. 
\end{proof}

\begin{lemma} \label{lem:BV_UpperBound}
Let $f:\mathbb C\to\mathbb R$ be a compactly supported function in $BV(\mathbb C,\sqrt{\Delta Q}).$ Then we have 
\begin{equation*}
\limsup_{n\to \infty} \frac1{\sqrt n} \int_{\mathbb C} \int_{\mathbb C} |f(z)-f(w)| |\mathcal K_n(z,w)|^2 dA(z) dA(w)
\leq \frac{1}{2 \pi \sqrt{\pi}} [f]_{BV}^{\sqrt{\Delta Q}}.
\end{equation*}
\end{lemma}

\begin{proof}
Let $f_j$ be smooth functions such that
$f_j\to f$ in $L^1$ pointwise and such that
$$\int_{\mathbb C}|\nabla f_j|\sqrt{\Delta Q}\to  [f]_{BV}^{\sqrt{\Delta Q}}.$$

By Fatou's lemma, Lemma \ref{lem:epsilonDelta}, Lemma \ref{lem:ineq_middle} and Lemma\ref{eq:fSmoothUpperBound}
\begin{align*}
  \frac1{\sqrt n}\int_{\mathbb C^2} &  |f(z)-f(w)| |\mathcal K_n(z,w)|^2 dA(z) dA(w)
  \\
  \le
&
\frac1{\sqrt n} \liminf_{j\to \infty}\int_{\mathbb C^2}  |f_j(z)-f_j(w)| |\mathcal K_n(z,w)|^2 dA(z) dA(w)  
  \\
  \le
&
\liminf_{j\to \infty} \frac1{\sqrt n}
\int_\Omega \int_{B(w,\delta_n)} |f_j(z)-f_j(w)| |\mathcal K_n(z,w)|^2 dA(z) dA(w)+\lVert f_j\rVert_{L^1} \mathcal O(n^{2-\epsilon \delta})
  \\
  =
&
\liminf_{j\to \infty} \frac{1+\mathcal O(1/\sqrt n)}{\sqrt n}
\int_\Omega \mathcal K_n(w,w) \int_{B(0, \delta_n \sqrt{n\Delta Q(w)})} \left|f_j\left(w+\frac{z}{\sqrt{n \Delta Q(w)}}\right)-f_j(w)\right| \\
& \hspace{5cm}e^{-|z|^2} dA(z) dA(w)+
 \lVert f_j\rVert_{L^1} \mathcal O(n^{2-\epsilon \delta})
  \\
  \le
&
\liminf_{j\to \infty}
\Bigl(\frac{1}{2 \pi \sqrt\pi}+\mathcal O(1/\sqrt n)\Bigr) \lVert \nabla f_j \sqrt{\Delta Q}\rVert_{L^1}+\lVert f_j\rVert_{L^1} \mathcal O(n^{2-\epsilon \delta})\\
= &
\Bigl(\frac{1}{2 \pi \sqrt\pi}+\mathcal O(1/\sqrt n)\Bigr)[f]_{BV}^{\sqrt{\Delta Q}}+\lVert f\rVert_{L^1} \mathcal O(n^{2-\epsilon \delta}),
\end{align*}
where the constant implied by the first $\mathcal O$ term is independent of $f$ (and $f_j$) and the constant implied by the second $\mathcal O$ term depends only on $K$ and $Q$. Taking the limsup, we arrive at the result.
\end{proof}

Recalling \eqref{eq:BV1AdHA}, we have the following immediate corollary of Lemma \ref{lem:BV_UpperBound}.

\begin{corollary} \label{eq:thmMainUpperBound}
Consider a random normal matrix model with a potential $Q$ that is $C^2$, real analytic on an open neighborhood of $S_Q$ and $\Delta Q>0$ on $S_Q$. 
Let $A\subset \mathbb C$ be a Borel set with finite perimeter. Then we have 
\begin{align} \label{eq:corMainThm}
\operatorname{Var }N_A^{(n)} \leq \frac{\sqrt n}{2\pi\sqrt\pi}\int_{\partial_* A} \sqrt{\Delta Q(z)} d\mathcal H^1(z)+\mathcal O(1),
\end{align}
uniformly for $n=1,2,\ldots$, where $\partial_* A$ is the measure theoretic boundary of $A$.\\
Given a compact set $K\subset \mathring S_Q$, \eqref{eq:mainThm} holds uniformly for $n=1,2,\ldots$ and all $A\subset K$.
\end{corollary}

Given $f \in BV(\mathbb C,\sqrt{\Delta Q})$ the regularization from the previous result is given by
$$f_j(z)=\int_{\mathbb C}f(z-u)\varphi_j(u)dA(u)$$
where $\varphi_j$ is a mollifier compactly supported in $B_{1/j}$, where $B_r:= B(0,r)$.

\begin{lemma} \label{lem:ineq_regul}
Let $f\in BV$, with compact support contained in the bulk $\Omega$, and let $f_j$ be a regularization as described above. Then for $j$ large enough
\begin{multline*}
\int_{\mathbb C^2} |f_j(z)-f_j(w)| |\mathcal K_n(z,w)|^2 dA(z) dA(w) \\
\le \int_{\mathbb C^2} |f(z)-f(w)| |\mathcal K_n(z,w)|^2 dA(z) dA(w)  
+  o(\sqrt n) + \mathcal O(\sqrt n j^{-1}),
\end{multline*}
where the term in $o(n)$ is independent of $j$.
\end{lemma}

\begin{proof}
Define
$$\mathcal I_h[f](z)=|f(z+h)-f(z)|.$$
From Lemma \ref{lem:epsilonDelta}, Lemma \ref{lem:ineq_middle} and Lemma \ref{eq:fSmoothUpperBound} we get that
\begin{align*}
& \int_{\mathbb C^2}  |f_j(z)-f_j(w)| |\mathcal K_n(z,w)|^2 dA(z) dA(w)
\\
&
=
\int_\Omega \mathcal K_n(w,w) \int_{B_{\delta_n \sqrt{n\Delta Q(w)}}} 
\mathcal  I_{\frac{z}{\sqrt{n \Delta Q(w)}}}[f_j](w) e^{-|z|^2} dA(z) dA(w)+o(\sqrt{n})    
\\
&
\le 
\int_\Omega \mathcal K_n(w,w) \int_{B_{\delta_n \sqrt{n\Delta Q(w)})}}\int_{B_{1/j}}\phi_j(u) 
\mathcal  I_{\frac{z}{\sqrt{n \Delta Q(w)}}}[f](w-u) e^{-|z|^2} dA(u) dA(z) dA(w)+o(\sqrt{n})    
\\
&
=
\int_\Omega \mathcal K_n(w,w) \Delta Q(w) \int_{B_{ \delta \log n}} \int_{B_{1/j}}\phi_j(u) 
\mathcal  I_{\frac{z}{\sqrt{n }}}[f](w-u)
 e^{-\Delta(w)|z|^2} dA(u) dA(z) dA(w)+o(\sqrt{n}) 
\\
&
=
 \int_{B_{\delta \log n}} \int_{B_{1/j}}\phi_j(u) \int_\Omega \mathcal K_n(w,w) \Delta Q(w) 
 \mathcal  I_{\frac{z}{\sqrt{n }}}[f](w-u)
 e^{-\Delta Q(w)|z|^2} dA(w) dA(u) dA(z)+o(\sqrt{n}) 
 \\
&
\le
 \int_{B_{\delta \log n}} \int_{B_{1/j}}\phi_j(u) \int_{\widetilde\Omega} \mathcal K_n(u+v,u+v) \Delta Q(u+v) 
 \mathcal  I_{\frac{z}{\sqrt{n }}}[f](v)
 e^{-\Delta Q(u+v)|z|^2} dA(v) dA(u) dA(z)+o(\sqrt{n}) ,
 \end{align*}
where the $o$ terms can be picked independently of $j$ and $\widetilde\Omega$ is a slightly bigger domain that contains all $1/j$-neighborhoods of $\Omega$. 
 Now we use the expansion for $|u|<j^{-1}$
 \begin{align*}
 \mathcal K_n(u+v,u+v) & = \mathcal K_n(v,v)+O(j^{-1}),
 \\
 \Delta Q(u+v) & =\Delta Q(v)+O(j^{-1})
  \end{align*}
  Since $|e^x-1| \le |x| e^{|x|}$ we have
  \[
  e^{-\Delta Q(u+v)|z|^2}  = e^{-\Delta Q(v)|z|^2}(1+ |z|^2\mathcal O(1/j)e^{\mathcal O(1/j)|z|^2)}).
 \]

The main error term is 
 $$\int_{B(0, \delta \log n)} \int_{B_{1/j}}\phi_j(u) \int_{\widetilde\Omega}  K_n(v,v) \Delta Q(v) \mathcal  I_{\frac{z}{\sqrt{n }}}[f](v)  e^{-\Delta Q(v)|z|^2}e^{\mathcal O(|z|^2/j)}|z|^2 \mathcal O(1/j)   dA(v) dA(u) dA(z)
$$
which is bounded, for some $c>0$, and $j$ big enough by 
 \begin{align*}
\int_{B(0, \delta \log n)} & \int_{B_{1/j}}\phi_j(u) n \int_{\widetilde\Omega}   \mathcal  I_{\frac{z}{\sqrt{n }}}[f](v)  e^{-c |z|^2}e^{\mathcal O(|z|^2/j)}|z|^2 \mathcal O(1/j)   dA(v) dA(u) dA(z)
\\
&
\lesssim  
[f]_{BV}^{\sqrt{\Delta Q}} \sqrt{n} j^{-1}
\int_{B(0, \delta \log n)} |z|^3  e^{-\frac{c}{2} |z|^2}  dA(z)=  [f]_{BV}^{\sqrt{\Delta Q}}  \mathcal O(\sqrt{n}j^{-1}),
\end{align*}
where we used that for a function $f$ of bounded variation
$$\| f(\cdot +h )-f\|_{L^1}\lesssim |h| [f]_{BV}^{\sqrt{\Delta Q}}.$$

Finally we get
\begin{align*}
& \int_{\mathbb C^2}  |f_j(z)-f_j(w)| |\mathcal K_n(z,w)|^2 dA(z) dA(w)
\\
&
 =
 \int_{B(0, \delta \log n)} \int_{B_{1/j}}\phi_j(u) \int_{\widetilde\Omega}   \mathcal K_n(v,v) \Delta Q(v)  \mathcal  I_{\frac{z}{\sqrt{n }}}[f](v) e^{-\Delta Q(v)|z|^2}  dA(v) dA(u) dA(z)+o(\sqrt{n})+\mathcal O(\sqrt nj^{-1})
  \\
&
 =
 \int_{B(0, \delta \log n)}  \int_{\widetilde\Omega}   \mathcal K_n(v,v) \Delta Q(v)  \mathcal  I_{\frac{z}{\sqrt{n }}}[f](v) e^{-\Delta Q(v)|z|^2}  dA(v)  dA(z)+o(\sqrt{n})+\mathcal O(\sqrt nj^{-1})
   \\
&
 =
 \int_{\mathbb C^2}  |f(z)-f(w)| |\mathcal K_n(z,w)|^2 dA(z) dA(w)+o(\sqrt{n})+\mathcal O(\sqrt nj^{-1}).
\end{align*}

\end{proof}

\begin{proof}[Proof of Theorem \ref{thm:main} and Proposition \ref{teo:BV}.]
By the equality for smooth functions in Corollary \ref{cor:mainThmSmooth} and Lemma \ref{lem:ineq_regul} we have 
\begin{align*}
  \frac{1}{2 \pi \sqrt{\pi}} [f_j]_{BV}^{\sqrt{\Delta Q}} & =\lim_{n\to \infty}
\frac1{\sqrt{n}}\int_{\mathbb C^2} |f_j(z)-f_j(w)| |\mathcal K_n(z,w)|^2 dA(z) dA(w)
\\
&
\le \limsup_{n\to \infty} \int_{\mathbb C^2} |f(z)-f(w)| |\mathcal K_n(z,w)|^2 dA(z) dA(w)  
+\mathcal O(j^{-1})
\end{align*}
taking the limit in $j$ we get now 
$$\frac{1}{2 \pi \sqrt{\pi}} [f]_{BV}^{\sqrt{\Delta Q}}\le \limsup_{n\to \infty} \int_{\mathbb C^2} |f(z)-f(w)| |\mathcal K_n(z,w)|^2 dA(z) dA(w)$$
and combined with Lemma \ref{lem:BV_UpperBound}, we arrive at Proposition \ref{teo:BV}.\\
Now Theorem \ref{thm:main} follows directly from Proposition \ref{teo:BV} and \eqref{eq:BV1AdHA}.
\end{proof}

\begin{remark}
    It should be mentioned that an approach similar in spirit has recently been carried out by Le Doussal and Schehr\cite{LDSc} in the rotationally symmetric case. The variance \cite{AkByEb} and higher order cumulants \cite{BDMS} in that setting were already well-understood, but their contribution is to describe explicitly how the transition from smooth functions to indicator functions happens. Without ever mentioning functions of bounded variation, they explicitly construct the smooth functions $f_j$ that approximate the indicator function (although they take a simultaneous limit $j=n\to\infty$). 
\end{remark}

\section{Proof of the edge theorem} \label{sec:proofDilation}

In this section we define
\begin{align*}
    \epsilon_n = M\sqrt{\log n},
\end{align*}
where $M>0$ is some constant that we shall take "large enough" if necessary. It will be somewhat more practical to prove Theorem \ref{thm:dilationNumberVar} with $\delta$ replaced by $\sqrt 2\delta$. 
The number variance for the set $A_n(\sqrt 2\delta)$, as defined in \eqref{eq:defAn} is given by the general formula
\begin{align*}
    \operatorname{Var} N_{A_n(\sqrt 2\delta)}^{(n)} = \int_{A_n(\sqrt 2\delta)} \int_{A_n(\sqrt 2\delta)^c} |\mathcal K_n(z,w)|^2 dA(z) dA(w).
\end{align*}
We claim that this integral is dominated by the integration region where $z$ and $w$ are both close to $\partial S_Q$ and close to each other. We have some freedom in our choice of "close", but it shall suffice to mean of order $\mathcal O(\sqrt\frac{\log n}{n})$ for $z-w$. Thus we claim that the number variance is well-approximated by $\pi^{-2}$ times the integral
\begin{multline*}
    I_n(\delta,Q) =\int_{\partial S_Q} \int_{\substack{\partial S_Q,\\ d(z_0,w_0)\leq \sqrt\frac{\log n}{n}}} \int_{-\epsilon_n}^\delta \int_{\delta}^{\epsilon_n}
        \frac{|\phi'(z_0)||\phi'(w_0)|}{n\sqrt{\Delta Q(z_0)\Delta Q(w_0)}}\times\\
        \left|\mathcal K_n\left(z_0+\frac{\vec n(z_0) \xi}{\sqrt{n \Delta Q(z_0)}},w_0+\frac{\vec n(w_0) \eta}{\sqrt{n \Delta Q(w_0)}}\right)\right|^2 d\xi d\eta d\mathcal H^1(w_0) d\mathcal H^1(w_0).
\end{multline*}
Here $d(z_0,w_0)$ is proportional to $|z_0-w_0|$ and is defined as
\begin{align*}
    d(z_0,w_0) = |\log(\phi(z_0) \overline{\phi(w_0)})|.
\end{align*}
We used here that the Jacobian of the transformation $z=h_n(z_0,\xi)$ with $h_n$ as defined in \eqref{eq:defhn} is given by
\begin{align*}
    \left|\frac{\partial(\Re z,\Im z)}{\partial(z_0,\xi)}\right| = \frac{|\phi'(z_0)|}{\sqrt{2n\Delta Q(z_0)}} \left(1-\frac{\kappa(z_0) \xi}{\sqrt{2n\Delta Q(z_0)}}\right)
    = \frac{|\phi'(z_0)|}{\sqrt{2n\Delta Q(z_0)}}+\mathcal O\left(\frac{\epsilon_n}{n}\right),
\end{align*}
where $\kappa(z_0)$ is the curvature at $z_0$, and the $\mathcal O(\frac{\epsilon_n}{n})$ term is already ignored in the definition of $I_n(\delta,Q)$. 
By Lemma \ref{lem:keyLemma} we have a good understanding of the behavior of the integrand. We can use it to prove the following lemma.

\begin{lemma} \label{lem:6.1}
    We have as $n\to\infty$ that
    \begin{multline*}
        \frac1{\sqrt n}I_n(\delta,Q)=(\frac14+\mathcal O(\frac{\log^3n}{\sqrt n}))
        \int_{\partial S_Q} \sqrt{\Delta Q(z_0)} |\phi'(z_0)|\int_{-\epsilon_n}^\delta \int_{\delta}^{\epsilon_n} 
        \int_{-|\phi'(z_0)|^{-1}\sqrt{\Delta Q(z_0)\log n}}^{|\phi'(z_0)|^{-1}\sqrt{\Delta Q(z_0)\log n}} \\
        e^{-(\xi-\eta)^2-\theta^2} \left|\operatorname{erfc}\left(\frac{\xi+\eta+i\theta}{\sqrt 2}\right)\right|^2 d\theta d\xi d\eta d\mathcal H^1(z_0)
    \end{multline*}
\end{lemma}

\begin{proof}
    Using Lemma \ref{lem:keyLemma} we have
    \begin{multline*}
    I_n(\delta,Q)=n\int_{\partial S_Q} \Delta Q(z_0)\int_{\substack{\partial S_Q,\\ d(z_0,w_0)\leq \sqrt\frac{\log n}{n}}} \int_{-\epsilon_n}^\delta \int_{\delta}^{\epsilon_n}
        (\frac14+\mathcal O(\frac{\log^3 n}{\sqrt n}))\\
        e^{-(\xi-\eta)^2+n\Delta Q(z_0)\frac{(\log(\phi(z_0)\overline{\phi(w_0)}))^2}{|\phi'(z_0)|^2}}
         \left|\operatorname{erfc}\left(\frac{\xi+\eta}{\sqrt 2}+\sqrt{n\Delta Q(z_0)}\frac{\log(\phi(z_0)\overline{\phi(w_0)})}{\sqrt 2|\phi'(z_0)|}\right)\right|^2\\
         d\xi d\eta |\phi'(w_0)| d\mathcal H^1(w_0) 
         |\phi'(z_0)| d\mathcal H^1(w_0),
\end{multline*}
where we already used that $\sqrt{\Delta Q(z_0)\Delta Q(w_0)}=\Delta Q(z_0)+\mathcal O(\sqrt\frac{\log n}{n})$. 
A substitution of integration variables $\phi(z_0)=e^{it}$ and $\phi(w_0)=e^{is}$ gives
\begin{multline*}
    I_n(\delta,Q) = (\frac14+\mathcal O(\frac{\log^3 n}{\sqrt n}))\int_{-\pi}^\pi \int_{t-\sqrt\frac{\log n}{n}}^{t+\sqrt\frac{\log n}{n}} \int_{-\epsilon_n}^\delta \int_\delta^{\epsilon_n} e^{-(\xi-\eta)^2-n\Delta Q(z_0)\frac{(s-t)^2}{|\phi'(z_0)|^2}}\times\\
    \left|\operatorname{erfc}\left(\frac{\xi+\eta}{\sqrt 2}+i\sqrt {n\Delta Q(z_0)}\frac{t-s}{\sqrt 2|\phi'(z_0)|}\right)\right|^2 d\xi d\eta ds dt
\end{multline*}
and this yields the result after substitution $s\to\theta=\sqrt {n\Delta Q(z_0)}\frac{t-s}{|\phi'(z_0)|}$ and the observation that $|\phi'(w_0)|^{-1}=|\phi'(z_0)|^{-1}+\mathcal O(\sqrt\frac{\log n}{n})$, and transforming $t$ back to $z$. 
\end{proof}

\begin{lemma} \label{lem:6.2}
    We have 
    \begin{multline*}
    \lim_{n\to\infty}\int_{-\epsilon_n}^\delta \int_{\delta}^{\epsilon_n} 
        \int_{-|\phi'(z_0)|^{-1}\sqrt{\Delta Q(z_0)\log n}}^{|\phi'(z_0)|^{-1}\sqrt{\Delta Q(z_0)\log n}} 
        e^{-(\xi-\eta)^2-\theta^2} \left|\operatorname{erfc}\left(\frac{\xi+\eta+i\theta}{\sqrt 2}\right)\right|^2 d\theta d\xi d\eta\\
        =
        \frac{\pi}{\sqrt 2} \int_{\sqrt 2\delta}^\infty \operatorname{erfc}(t)
        \operatorname{erfc}(-t) dt.
        \end{multline*}
\end{lemma}

\begin{proof}
It follows from the asymptotic behavior of the complementary error function (e.g., see 8.357 in \cite{GrRh}) that
    \begin{multline} \label{eq:erfcIntegralEst1}
        \int_{|\phi'(z_0)|^{-1}\sqrt{\Delta Q(z_0)\log n}}^\infty 
        e^{-\theta^2} \left|\operatorname{erfc}\left(\frac{\xi+\eta+i\theta}{\sqrt 2}\right)\right|^2 d\theta\\
        = \int_{|\phi'(z_0)|^{-1}\sqrt{\Delta Q(z_0)\log n}}^\infty \mathcal O\left(\frac{e^{-(\xi+\eta)^2}}{(\xi+\eta)^2+\theta^2}\right) d\theta
        = e^{-(\xi+\eta)^2}\mathcal O\left(\frac{|\phi'(z_0)|}{\sqrt{\Delta Q(z_0)\log n}}\right),
    \end{multline}
    where the implied constant is uniform for $\xi, \eta=\mathcal O(\sqrt{\log n})$. Denoting $x_\pm=\frac{x\pm y}{\sqrt 2}$ and by $H$ the Heaviside step function (say, with $H(0)=\frac12$), we get
    \begin{multline} \label{eq:erfcIntegralEst2}
        \int_{-\infty}^\infty e^{-\theta^2} \left|\operatorname{erfc}\left(\frac{\xi+\eta+i\theta}{\sqrt 2}\right)\right|^2 d\theta\\
        = \frac{4}{\pi}\int_{-\infty}^\infty e^{-\theta^2}\int_{-\infty}^\infty \int_{-\infty}^\infty e^{-(x-\frac{\xi+\eta+i\theta}{\sqrt 2})^2} e^{-(y-\frac{\xi+\eta-i\theta}{\sqrt 2})^2} dx dy\\
        = \frac{4}{\pi} \int_{-\infty}^\infty \int_{-\infty}^\infty \int_{|x_-|}^\infty e^{-(x_+-\xi-\eta)^2} e^{-x_-^2} e^{2i\theta x_-} dx_+ dx_- d\theta\\
        =\frac{4}{\sqrt\pi} \int_{-\infty}^\infty \int_{-\infty}^\infty H(x_-) \operatorname{erfc}(x_-+\xi+\eta) e^{-x_-^2} e^{2i\theta x_-} dx_- d\theta
        = 2\sqrt\pi \operatorname{erfc}(\xi+\eta),
    \end{multline}
    where the last step can be proved with the Fourier inversion formula, since $x\mapsto H(x)\operatorname{erfc}(x+\xi+\eta) e^{-x^2}$ is absolutely integrable and piecewise smooth. Hence
\begin{multline*}
    \int_{-\epsilon_n}^\delta \int_{\delta}^{\epsilon_n}\int_{-|\phi'(z_0)|^{-1}\sqrt{\Delta Q(z_0)\log n}}^{|\phi'(z_0)|^{-1}\sqrt{\Delta Q(z_0)\log n}} 
        e^{-\theta^2} \left|\mathrm{erfc}\left(\frac{\xi+\eta+i\theta}{\sqrt 2}\right)\right|^2 d\theta d\xi d\eta\\
        = 2\sqrt\pi\int_{-\epsilon_n}^\delta \int_{\delta}^{\epsilon_n} e^{-(\xi-\eta)^2}\operatorname{erfc}(\xi+\eta) d\xi d\eta
        +\mathcal O\left(\frac1{\sqrt{\log n}}\right)\\
        = 2\sqrt\pi \int_0^{\epsilon_n+\delta} \int_0^{\epsilon_n-\delta} e^{-(\xi+\eta)^2} \operatorname{erfc}(\xi-\eta+2\delta) d\xi d\eta
        +\mathcal O\left(\frac1{\sqrt{\log n}}\right),
\end{multline*}
    as $n\to\infty$, uniformly for $z_0\in\partial S_Q$ and $\delta$ in compact sets, where we used \eqref{eq:erfcIntegralEst1} and \eqref{eq:erfcIntegralEst2} to go from the first to the second line. Let us consider the integral with $\epsilon_n$ replaced by $\infty$ (the difference is negligible). We have 
    \begin{multline*}
        2\sqrt\pi\frac{\partial}{\partial t} \int_0^\infty \int_{0}^\infty e^{-(\xi+\eta)^2}\operatorname{erfc}(\xi-\eta+2t) d\xi d\eta\\
        = -8 e^{-4t^2} \int_0^\infty e^{-2\xi^2-4t \xi} \int_0^\infty e^{-2\eta^2+4t \eta} d\xi d\eta
        = -  \pi\operatorname{erfc}(\sqrt 2t) \operatorname{erfc}(-\sqrt 2t).
        \end{multline*}
    We obtain the result after integrating $t$ from $\delta$ to $\infty$, and rescaling $t$.
\end{proof}

What is left is to prove that the remaining integration region of our original integral defining the number variance is negligible. There is a version of \eqref{eq:Knzw} (Proposition 3.6 in \cite{AmHeMa2}) that essentially tells us that points at distance $M\sqrt\frac{\log n}{n}$ from the boundary are negligible for our integral, if we pick $M$ large enough here and in the definition of $\epsilon_n$ (see also \cite{Ameur}). We thus have to focus on the integral
\begin{multline*}
    \tilde I_n(\delta,Q) =\int_{\partial S_Q} \int_{\substack{\partial S_Q,\\ d(z_0,w_0)\geq \frac{\log n}{\sqrt n}}} \int_{-\epsilon_n}^\delta \int_{\delta}^{\epsilon_n}
        \frac1{n\sqrt{\Delta Q(z_0)\Delta Q(w_0)}}\times\\
        \left|\mathcal K_n\left(z_0+\frac{\vec n(z_0)\xi}{\sqrt{n \Delta Q(z_0)}},w_0+\frac{\vec n(w_0)\eta}{\sqrt{n \Delta Q(w_0)}}\right)\right|^2 d\xi d\eta d\mathcal H^1(w_0) d\mathcal H^1(w_0).
\end{multline*}

\begin{lemma} \label{lem:negli}
    We have as $n\to\infty$ that
    \begin{align*}
        \frac1{\sqrt n} \tilde I_n(\delta,Q) = \mathcal O(\frac1{\sqrt{\log n}}).
    \end{align*}
\end{lemma}

\begin{proof}
    Using Lemma \ref{lem:keyLemma2}, and a rescaling $(\xi,\eta)\to \sqrt n(\xi,\eta)$ followed by Laplace's method, we infer that as $n\to\infty$
    \begin{align*}
        \tilde I_n(\delta,Q) &\lesssim \int_{\partial S_Q} \int_{\substack{\partial S_Q,\\ d(z,w)\geq \sqrt\frac{\log n}{n}}} \frac{d\mathcal H^1(w_0) d\mathcal H^1(z_0)}{|\phi(z_0)\overline{\phi(w_0)}-1|^2} \\
        &\lesssim \int_{-\pi}^\pi \int_{-\pi, |\theta_1-\theta_2|\geq \sqrt\frac{\log n}{n}}^\pi \frac{d\theta_2 d\theta_1}{|e^{i(\theta_1-\theta_2)}-1|^2}\\
        &= \int_{-\pi}^\pi \mathcal O(\cot(\frac12\sqrt\frac{\log n}{n})) d\theta_1 
        = \mathcal O\left(\sqrt\frac{n}{\log n}\right),
    \end{align*}
    uniformly for $\delta$ in compact sets.
\end{proof}

\begin{proof}[Proof of Theorem \ref{thm:dilationNumberVar}]
    As mentioned above, e.g., by \cite{AmHeMa2} or \cite{Ameur}, the contribution to the integral defining the variance \eqref{eq:numberVarCorKerIntA} is exponentially small for any set of points $(z, w)$ where at least one of $z, w$ has distance of order $\varepsilon_n/\sqrt n$ to the edge. 
    By Lemma \ref{lem:negli} the contribution to the integral of the remaining points that also satisfy the condition $d(z_0,w_0)\geq \frac{\log n}{\sqrt n}$ is of order $\mathcal O(\sqrt{n/\log n})$. In conclusion, we have
    \[
    \frac1{\sqrt n}\operatorname{Var} N_{A_n(\sqrt 2\delta)}^{(n)} = \frac1{\pi^2} I_n(\delta,Q)+\mathcal O\left(\frac1{\sqrt{\log n}}\right)
    \]
    as $n\to\infty$, uniformly for $\delta$ in compact sets. Finally, we use Lemma \ref{lem:6.1} and Lemma \ref{lem:6.2} to arrive at
    \begin{align*}
    \frac1{\sqrt n}\operatorname{Var} N_{A_n(\sqrt 2\delta)}^{(n)} &= \frac1{4\pi^2} \frac{\pi}{\sqrt 2} \int_{\sqrt 2\delta}^{\infty} \operatorname{erfc}(t) \operatorname{erfc}(-t) \, dt +\mathcal O\left(\frac1{\sqrt{\log n}}\right)\\
    &= \frac1{2\pi\sqrt\pi} f(\sqrt 2 \delta)
    + \mathcal O\left(\frac1{\sqrt{\log n}}\right)
    \end{align*}
    as $n\to\infty$, uniformly for $\delta$ in compact sets.
\end{proof}

\section*{Acknowledgments}
LM thanks Aron Wennman for useful discussions.

\end{document}